\documentclass[11pt,a4paper]{amsart}
\usepackage{amscd}
\usepackage{amssymb}
\usepackage[centering,text={15.5cm,22cm}]{geometry}
\usepackage{eqnarray}
\usepackage{graphicx,color}
\usepackage{tikz}
\usepackage[all]{xy}
\usepackage{hyperref}
\usepackage{enumerate}
\allowdisplaybreaks

\definecolor{shadecolor}{rgb}{1,0.9,0.7}

\newtheorem{Theorem}{Theorem}[section]
\newtheorem{Lemma}[Theorem]{Lemma}
\newtheorem{Lemma-definition}[Theorem]{Lemma-Definition}
\newtheorem{Proposition}[Theorem]{Proposition}
\newtheorem{Corollary}[Theorem]{Corollary}

\newtheorem{thm}[Theorem]{Theorem}
\newtheorem{Conjecture}[Theorem]{Conjecture}

\theoremstyle{definition}

\newtheorem{definition}[Theorem]{Definition}
\newtheorem{construction}[Theorem]{Construction}

\newtheorem{example}[Theorem]{Example}

\newtheorem{remark}[Theorem]{Remark}

\newtheorem{openque}[Theorem]{Open Question}

\theoremstyle{remark}

\numberwithin{equation}{section}
\numberwithin{figure}{section}

\newcommand{\NN} {\mathbb{N}}

\newcommand{\ZZ} {\mathbb{Z}}
\newcommand{\QQ} {\mathbb{Q}}
\newcommand{\RR} {\mathbb{R}}
\newcommand{\CC} {\mathbb{C}}

\newcommand{\PP} {\mathbb{P}}
\renewcommand{\AA} {\mathbb{A}}

\newcommand {\shB}  {\mathcal{B}}
\newcommand {\shC}  {\mathcal{C}}
\newcommand {\shD}  {\mathcal{D}}

\newcommand {\shE}  {\mathcal{E}}

\newcommand {\shM}  {\mathcal{M}}

\newcommand {\shN}  {\mathcal{N}}
\newcommand {\shO}  {\mathcal{O}}

\newcommand {\shS}  {\mathcal{S}}

\newcommand {\shP}  {\mathcal{P}}

\newcommand {\shZ}  {\mathcal{Z}}

\renewcommand {\ker } {\operatorname{ker}}

\newcommand {\ol} {\overline}
\newcommand {\ord}  {\operatorname{ord}}

\newcommand {\Pic}  {\operatorname{Pic}}

\newcommand {\Spec} {\operatorname{Spec}}

\newcommand{\ndiv}{\hspace{-4pt}\not|\hspace{2pt}}

\DeclareMathOperator {\hhh}{H}
\DeclareMathOperator {\GW}{\mathcal{GW}}

\DeclareMathOperator {\vdim} {vdim}
\DeclareMathOperator {\residue} {res}

\newcommand{\prim}{\mathrm{prim}}
\newcommand{\integ}{\mathrm{int}}

\newcommand{\ic}{\mathbb C}
\newcommand{\iz}{\mathbb Z}
\newcommand{\iq}{\mathbb Q}

\DeclareMathOperator{\gl}{GL}
\DeclareMathOperator{\hilb}{Hilb}
\DeclareMathOperator{\dt}{DT}

\def\mydate{\ifcase\month \or January\or February\or March\or
April\or May\or June\or July\or August\or September\or October\or 
November\or December\fi \space\number\day,\space\number\year}

\newcommand{\Q}{\mathbb Q}

\newcommand{\IP}{\mathbb P}

\newcommand{\ptwo}{\IP^2}
\newcommand{\pone}{\IP^1}

\newcommand{\mc}{\mathcal}

\DeclareMathOperator{\HHH}{H}
\DeclareMathOperator{\hh}{h}

\DeclareMathOperator{\diff}{d}

\DeclareMathOperator{\rest}{res}
\DeclareMathOperator{\tot}{Tot}
\DeclareMathOperator{\contr}{Contr}
\DeclareMathOperator{\mmm}{M}
\DeclareMathOperator{\obstr}{Ob}
\DeclareMathOperator{\pic}{Pic}

\DeclareMathOperator{\bps}{BPS}

\newcommand{\mse}{\ol{\mmm}_{\beta}(S,E)}
\newcommand{\msep}{\ol{\mmm}^P_{\beta}(S,E)}

\newcommand{\nse}{\shN_\beta(S,E)}

\newcommand{\nsep}{\shN^P_\beta(S,E)}

\newcommand{\multc}{e(\ol{\pic}^0(C))}

\newcommand{\rcjp}{\ol{\pic}^0(\beta,P)}

\newcommand{\xyR}[1]{
  \xydef@\xymatrixrowsep@{#1}}
\newcommand{\xyC}[1]{
  \xydef@\xymatrixcolsep@{#1}}

\begin{document}

\title
[Log BPS numbers]
{Log BPS numbers of log Calabi-Yau surfaces}

\author[J. Choi]{Jinwon Choi}
\address{Department of Mathematics \& Research Institute of Natural Sciences, Sookmyung Women's University, Cheongpa-ro 47-gil 100, Youngsan-gu, Seoul 04310, Republic of Korea}
\email{jwchoi@sookmyung.ac.kr}

\author[M. van Garrel]{Michel van Garrel}
\address{Mathematics Institute, Zeeman Building, University of Warwick, Coventry CV4 7AL, UK}
\email{michel.van-garrel@warwick.ac.uk}

\author[S. Katz]{Sheldon Katz}
\address{Department of Mathematics, MC-382, University of Illinois at Urbana-Champaign, Urbana, IL 61801, USA}
\email{katz@math.uiuc.edu}

\author[N. Takahashi]{Nobuyoshi Takahashi}
\address{Department of Mathematics, Graduate School of Science, Hiroshima University, 1-3-1 Kagamiyama, Higashi-Hiroshima, 739-8526 JAPAN}
\email{tkhsnbys@hiroshima-u.ac.jp}


\begin{abstract}
Let $(S,E)$ be a log Calabi-Yau surface pair with $E$ a smooth divisor. 
We define new conjecturally integer-valued counts of 
$\mathbb{A}^1$-curves in $(S,E)$. These \emph{log BPS numbers} are derived from genus 0 log Gromov-Witten invariants of maximal tangency along $E$ via a formula analogous to the multiple cover formula for disk counts.
A conjectural relationship to genus 0 local BPS numbers is described and verified for del Pezzo surfaces and curve classes of arithmetic genus up to 2. We state a number of conjectures and provide computational evidence.
\end{abstract}

\maketitle
\setcounter{tocdepth}{1}
\tableofcontents
\medskip

\section{Introduction}

Let $(S,E)$ be a \emph{log Calabi-Yau surface with a smooth divisor}, by which we shall mean that $S$ is a smooth projective surface and $E$ is a smooth effective anticanonical divisor on it. By the adjunction formula, each connected component of $E$ has genus 1. It is expected that $S\setminus E$ admits a Strominger-Yau-Zaslow (SYZ) special Lagrangian torus fibration \cite{SYZ}, with singular fibers away from $E$. When $S$ is del Pezzo, this is shown in \cite{CJL}. We are interested in counts of holomorphic disks in $S\setminus E$ with boundary ending on a SYZ fiber \emph{near} $E$. In analogy to \cite[Theorem 3.4]{GPS10} and \cite{Lin1,Lin2}, these are in turn predicted to correspond to counts of tropical disks in the relevant scattering diagrams of \cite{CPS}. The latter describe the wall-crossing automorphisms, cf.\ \cite[Proposition 5.4]{CPS}.

The SYZ mirror conjecture is successfully implemented in algebraic geometry in the Gross-Siebert program
\cite{GS11,CPS}, see also \cite{GS06,GHK,GHKS18,GS16,Gbook,GOR}.
Assume that $S$ is del Pezzo. The construction of the mirror for $(S,E)$ proceeds via the relevant scattering diagram as detailed in \cite{CPS}. In particular, the superpotential on the mirror is constructed via summing the monomials attached to broken lines.
In \cite[\S 5.4]{CPS} the wall-crossing functions are expressed as generating functions of Maslov index 0 tropical disks. They in turn are expected to be expressible as counts of $\mathbb{A}^1$-curves on $(S,E)$, which are rational curves in $S$ meeting $E$ in one point of maximal tangency. 
Their virtual definition is as the ($\QQ$-valued) genus 0 \emph{log} or \emph{relative Gromov-Witten invariants} \cite{Ga02,Li01,Li02,GS13,AC,Ch} of maximal tangency. When $S=\ptwo$, the correspondence between genus 0 log GW invariants of maximal tangency and tropical curves in the scattering diagram is established in \cite{Gra}.

In this paper, we explore how to obtain $\NN$-valued invariants out of these log Gromov-Witten (GW) invariants, which should be the underlying counts of (immersed) $\AA^1$-curves. In accordance with standard definitions, we call them \emph{log BPS numbers}.
The relationship between the log Gromov-Witten invariants of maximal tangency and the log BPS numbers is the formula of Definition \ref{def1}.
It is analogous to the multiple cover formula for genus 0 open Gromov-Witten invariants \cite{FOOO} and generalized DT invariants \cite{JS}.
In \cite[Proposition 6.1]{GPS10} (see \eqref{eq:BPS}), the authors compute the contribution of multiple covers over rigid relative maps to the relative GW invariants. This leads them to define \emph{relative BPS state counts}, which are shown to be integers for toric del Pezzo surfaces in \cite{ga number}. One motivation for this work is that 
the enumerative meaning of relative BPS state counts are not clear 
in the context of a smooth divisor. We make the connection of log BPS numbers with relative BPS state counts and loop quiver DT invariants in \S\ref{sec:loop}.

Apart from the connection to the Gross-Siebert program, the advantages of log BPS state counts are twofold. Firstly, generically they are weighted counts of curves (Proposition \ref{prop:logk3contr}). Secondly, they are conjecturally independent of the point of contact (Conjecture \ref{conj1}).

In this paper, we set up the theory of log BPS state counts. We state a number of conjectures, some of which we prove for arithmetic genus up to 2 in the case of del Pezzo surfaces. This paper is a continuation of \cite{CGKT} and uses results from \cite{CGKT3}. 

\subsection{Summary of results and conjectures}

We will always denote by $S$ a smooth projective surface and by $E$ a smooth divisor on $S$, and require additional conditions on $(S,E)$ as needed. Let $\beta\in\hhh_2(S,\ZZ)$ be a \emph{curve class}, by which we shall mean that $\beta$ can be represented by a nonempty one-dimensional subscheme. Assume also that $w:=\beta.E>0$. If the triple $(S,E,\beta)$ satisfies $(K_S+E).\beta=0$, following \cite{GPS10} we will say that $(S,E)$ is \emph{log Calabi-Yau with respect to} $\beta$.

Denote by $\mse$ the moduli space of maximally tangent genus 0 basic stable log maps to $(S,E)$ of class $\beta$ \cite{AC,Ch,GS13}. Typically we will require that $E$ is an elliptic curve, in which case the elements of $\mse$ are described in Corollary \ref{cor_description_log_maps}. $\mse$ admits a perfect obstruction theory, which is of virtual dimension $\mathrm{vdim}=-(K_S+E).\beta$, and yields a virtual fundamental class
\[
[\mse]^{\mathrm{vir}}\in\hhh_{2 \, \mathrm{vdim}}(\mse),
\]
as well as corresponding \emph{genus 0 log Gromov-Witten invariants of maximal tangency}
\[
\nse := \int_{[\mse]^{\mathrm{vir}}} 1 \in \QQ.
\]
Because of the virtual dimension of $\mse$, $\nse$ can be non-zero only if $(K_S+E).\beta=0$. 
We define the \emph{total} log BPS numbers for any such $(S,E,\beta)$.

Let $\hilb_\beta(S)$ be the Hilbert scheme of effective divisors on $S$ of class $\beta$ and let $\Pic^\beta(S)$ be the Picard variety  of isomorphism classes of line bundles $L$ on $S$ with
$c_1(L)$ Poincar\'e dual to $\beta$. $\Pic^\beta(S)$ is a complex torus of dimension $\mathrm{h}^{0,1}(S)$. Consider the Abel-Jacobi map $\hilb_\beta(S)\to \Pic^\beta(S)$. Then the fiber over $L$ is the linear system $|L|$.

Consider the set
\[
E(L) := \left\{P\in E \; \big{|} \;  L|_E \sim wP \right\}
\]
of possible points of contact of maximally tangent curves in $|L|$ with $E$.
Often we will assume that $E$ is an elliptic curve. Then $E(L)$
is a torsor for $\Pic^0(E)[w]\cong\ZZ/w\times\ZZ/w$, so that by choosing $0_E\in E(L)$, $E(L)$ is identified with the $w$-torsion points of $E$. This is proven in Lemma \ref{lem:ebetagroup} under the additional assumption that $S$ is a regular surface, i.e.\ a surface with irregularity $\hh^1(\shO_S)=0$, so that $\beta$ determines a unique Chow class. This assumption is not necessary, but we will often require it for a simpler exposition and then write $E(\beta)=E(L)$.

When $E(\beta)$ is finite, the log BPS numbers are obtained by fixing the point of contact $P\in E(\beta)$ and applying a suitable \emph{multiple cover} type formula. And the main conjecture (Conjecture \ref{conj1}) states that the definition is independent of the choice of $P\in E(\beta)$.

The different $P\in E(\beta)$ are classified according to their order with respect to a \emph{minimal} $0_E\in E(\beta)$ (Lemma \ref{lem:order}). As a rule of thumb, the lower the order, the more degenerate the stable log maps can be. On one extreme are points of order 1. Over such points, multiple covers of maximal degree may occur and the image cycle may decompose into a large number of irreducible components (see Corollary \ref{cor_description_log_maps}(3)).

On the opposite end are points $P\in E(\beta)$ of maximal order with respect to $0_E$, which we call $\beta$-primitive (Definition \ref{def:prim}). At these points, the image cycles are irreducible and the log BPS numbers are a weighted count of a finite number of maximally tangent rational curves (Propositions \ref{prop:logk3contr} and \ref{prop:finite}).

If $S$ is regular and $E$ is an elliptic curve,
Proposition \ref{prop:rat} states that under certain conditions 
$S$ has to be rational for log BPS numbers to be non-zero. 
Accordingly, in this introduction we mostly assume $S$ is rational hereafter. Note that if $S$ is rational and $E$ smooth anticanonical, then $E$ is necessarily an elliptic curve (after choosing a zero element).

Whenever $E(\beta)$ is finite,
\[
\mse  = \bigsqcup_{P\in E(\beta)} \msep
\]
and $\nse$ decomposes as a finite sum
\[
\nse = \sum_{P\in E(\beta)} \nsep
\]
according to the contributions from each $\msep$.
As is described in Corollary \ref{cor_description_log_maps}, see also Figure \ref{fig} and Section \ref{sec:evidencebig}, the relevant moduli space may contain both multiple covers and reducible image curves. Typically, this happens when $\beta=d\bar{\beta}$ and $P\in E(\bar{\beta})$, for a homology class $\bar{\beta}$ and $d>1$. In this case, $\msep$, and hence also $\nsep$, depends on the divisibility properties of $P$ (when one chooses the zero element $0_E\in E$ from $E(\bar{\beta})$) as illustrated in Section \ref{sec:evidencebig}. The motivation behind Definition \ref{def1} of log BPS numbers is to remove the dependence on $P$.

The next definition and conjecture constitute the principal novelty to curve counts of the present paper. They are motivated from discussions with Pierrick Bousseau, through considerations of open Gromov-Witten invariants. In the case of $S=\ptwo$, they were already formulated in \cite[Remark 4.11]{tak mult}.

\begin{definition}[See Definition \ref{multcover}]\label{def1}
Let $(S,E)$ be a log Calabi-Yau surface with respect to $\beta\in\hhh_2(S,\ZZ)$. The \emph{total log BPS number} $m^{tot}_\beta$ is defined implicitly via 
\[
\mc{N}_\beta(S,E) =  \sum_{k|\beta}
\frac{(-1)^{(k-1)w/k}}{k^2} \, m^{tot}_{\beta/k}.
\]

Assume that $S$ is rational, $E$ anticanonical and let $P\in E(\beta)$. The \emph{log BPS number at $P$}, $m^P_\beta$, is defined implicitly via 
\[
\mc{N}^P_\beta(S,E) =  \sum_{k|\beta}     
\frac{(-1)^{(k-1)w/k}}{k^2} \, m^P_{\beta/k},
\]
where we set $m^P_{\beta'}=0$ if $P\not\in E(\beta')$. 
We note that there is an inclusion $E(\beta/k)\subseteq E(\beta)$, 
so that $m^{tot}_\beta =\sum_{P\in E(\beta)}m^P_\beta$ holds. 
\end{definition}

Note that the number of nontrivial terms in the above formula for $\mc{N}^P_\beta(S,E)$ depends on the arithmetic properties of $P$ when $\beta=d\bar{\beta}$, $d>1$ . At $\beta$-primitive points, which are roughly points of the  maximal order, there will only be one term, whereas at points of lower order there will be many: see Section \ref{sec:evidencebig} for examples.

Writing $\mu$ for the M\"obius function, notice also that by M\"obius inversion
\[
m^{P}_\beta =  \sum_{k|\beta}  \; \frac{(-1)^{(k-1)w/k}}{k^2} \, \mu(k) \, \mathcal N^{P}_{\beta/k}(S,E),
\]
hence the $m^{P}_\beta$ are uniquely determined and similarly for the $m^{tot}_\beta$.

\begin{Conjecture}\label{conj_integrarity}
Let $(S,E)$ be a log Calabi-Yau surface with respect to $\beta\in\hhh_2(S,\ZZ)$. Then $m_\beta^{tot}\in\NN$. Additionally, if $S$ is rational and $E$ is elliptic, then $m_\beta^P\in\NN$ for $P\in E(\beta)$. 
\end{Conjecture}

For a del Pezzo surface and an anticanonical divisor, the above conjecture is a consequence of  Conjecture \ref{conj1}, Proposition \ref{prop:logk3contr} (1), Proposition \ref{prop:existence} and deformation invariance of  $\mc{N}^P_\beta(S,E)$.
Note that in general, for $P\neq P'\in E(\beta), \; \mc{N}^P_\beta(S,E)\neq \mc{N}^{P'}_\beta(S,E)$.

\begin{Conjecture}[See Conjecture \ref{conj:big}]\label{conj1}

Let $(S,E)$	be a rational log Calabi-Yau surface with smooth divisor and let $\beta\in\hhh_2(S,\ZZ)$. For all $P,P'\in E(\beta)$,
\[
m^P_\beta = m^{P'}_\beta.
\]
Equivalently, for all $P\in E(\beta)$,
\[
m^{tot}_\beta = w^2 \, m^P_\beta.
\]
\end{Conjecture}

Recall that for a del Pezzo surface $S$ other than $\ptwo$, its effective cone is generated by the line and conic classes (Definition \ref{def:lineconicclasses}). If $S\neq\pone\times\pone$, denote by $h$ the pullback of the hyperplane class. Our first main result, proven in Section \ref{sec:loop}, is:

\begin{Theorem}
\label{thm1}
Let $S$ be a Pezzo surface and let $E$ be a smooth anticanonical divisor. Let $\beta$ be a curve class, which is either a multiple of a line class, a multiple of a conic class, or, when $S\neq\pone\times\pone$, $dh$, for $1\leq d \leq 4$. Then Conjecture \ref{conj1} holds for $\beta$. In particular, the corresponding log BPS numbers at each point are calculated in Propositions \ref{prop:klineconic} and \ref{prop:calchyperplane}.
\end{Theorem}

After initial submission of this manuscript, a proof of Conjecture \ref{conj1} for $\ptwo$ in all degrees appeared in \cite{Bou19a,Bou19b}. The proof of \cite{Bou19a,Bou19b} proceeds via the scattering diagram of \cite{CPS} and uses \cite{Gra}. Our proof on the other hand proceeds by direct analysis of the moduli space. We use the calculation of contributions of maps with image consisting of 2 components of \cite{tak mult,CGKT3} and find a relationship with multiple cover contributions to log BPS numbers with quiver DT invariants (Proposition \ref{prop:loop}).

These new invariants exhibit a surprising connection to the \emph{local BPS invariants} $n_\beta$ \cite{GV1,GV2,PHodge,P3,BP,katz}. For the definition in general, see \eqref{formmult}. In the case of del Pezzo surfaces, there is the following equivalent definition from \cite{katz}.

\begin{definition}\label{defn:katz}
Assume that $S$ is a del Pezzo surface. 
Denote by $\mathcal{M}_\beta$ the moduli space of one-dimensional stable (with respect to $-K_S$) sheaves $F$ on $S$ with holomorphic Euler characteristic $\chi(F)=1$ and $[F] = \beta$. 
The \emph{genus 0 local BPS invariant} $n_\beta=n_\beta^0\in\ZZ$ of class $\beta$ is 
\[
n_\beta := (-1)^{w-1}e(\mathcal{M}_\beta),
\]
where $e(\cdot)$ denotes the (topological) Euler characteristic.\footnote{The sign is usually given as $(-1)^{\beta.\beta+1}$, which equals $(-1)^{w-1}$ by the adjuction formula.}
\end{definition}

Remarkably, see \eqref{eq:sum_log_in_local}, if $S$ is rational and $E$ is anticanonical and nef, Conjecture \ref{conj1} has the equivalent characterization that
for all $P\in E(\beta)$,
\begin{equation}\label{eq:nm}
n_\beta = (-1)^{w-1} \, w \, m^P_\beta.
\end{equation}

We introduce in Definition \ref{def:prim} the notion of $\beta$-primitive points for del Pezzo surfaces. At such a point $P$, $m_\beta^P=\nsep$ is an $\NN$-weighted count of maximally tangent rational curves and thus $m_\beta^P\in\NN$ is automatic (Proposition \ref{prop:logk3contr}(5)). There are no multiple covers or reducible image curves and the moduli space consists of a finite number of points. There is thus no need to consider the perfect obstruction theory. Instead, the contribution of each point to $m_\beta^P$ is simply given by its length, which is calculated in Proposition \ref{prop:logk3contr}(3).

Stating equation \eqref{eq:nm} only for $\beta$-primitive points leads to the following conjecture. It is a BPS version of the log-local principle put forward and proved in some cases in \cite{GGR}, and further developed in \cite{BBvG1,BBvG2,NR}, as well as in \cite{FTY,TY} in relation to orbifold Gromov-Witten theory.

\begin{Conjecture}[Log-local principle for BPS invariants]
\label{conj3}
Assume that $S$ is del Pezzo and $E$ anticanonical, and let $P\in E(\beta)$ be $\beta$-primitive. Then
\[
n_\beta = (-1)^{w-1} \, w \, m^P_\beta.
\]
\end{Conjecture}

Viewing $m^P_\beta$ as counts of disks in $S\setminus E$ bounding a SYZ fiber near $E$ and viewing $n_\beta$ as counts of rational curves in a (hypothetic) deformation of $\tot(K_S)$, Conjecture \ref{conj3} can be seen as an instance of \emph{open-closed duality}, see \S\ref{sec:open-closed}.

\begin{remark}
While we restrict to del Pezzo surfaces, the definition of primitive points makes sense for more general surfaces. Moreover, $n_\beta$ is can be defined when $w>0$. Consequently Conjecture \ref{conj3} is expected to hold in greater generality for curve classes with $w>0$.
As an illustration of this, let $X$ be a generalized del Pezzo surface of degree $d$ for $1\leq d\leq 9$, i.e.\ $X$ is a smooth projective surface with $-K_X$ big and nef and with $K_X^2=d$. For example, if $d=8$, $X$ could be the Hirzebruch surface $\mathbb{F}_{2}$. Other examples include the blowups of $\ptwo$ in $9-d$ points with at most 3 on a line and at most 6 on a conic.

Then $X$ can be deformed to a degree $d$ del Pezzo surface $S$. Denote by $\Gamma$ the cone spanned by the set of $(-2)$-curves on $X$. Then the effective cone of $X$ is the sum of the effective cone of $S$ and $\Gamma$, see e.g.\ \cite[Corollary 3.12]{DJT}. Consequently a $\beta$-primitive point on $X$ deforms to a $\beta$-primitive point on $S$, and whenever Conjecture \ref{conj3} holds for $S$, it also holds for $X$ with the appropriate identifications of curve classes. (If $\beta$ is not effective on $S$, the BPS numbers are $0$ anyway.)

That being said, we note that $\beta$-primitivity takes a particularly simple form for del Pezzo surfaces (Lemma \ref{lem:decomp}). We leave a careful analysis of $\beta$-primitivity for other surfaces to future work.

\end{remark}

If $S$ is del Pezzo, $E$ is anticanonical and $P\in E(\beta)$ is $\beta$-primitive, then by definition, each curve of class $\beta$ (and of any genus) which meets $E$ only at $P$ is integral. Among those, we denote by $M_{\beta,P}$ the curves that can arise as images of maximally tangent genus 0 stable log maps (Definition \ref{def:mbeta}).
It is a finite set (Proposition \ref{prop:finite}) consisting of integral rational curves of class $\beta$ maximally tangent to $E$ at $P$. For us, this includes the requirement that $C\in M_{\beta,P}$ has only 1 analytic branch on $E$ (otherwise the normalization map $n : \pone \to C$ could not be lifted to a maximally tangent stable log map).

At a $\beta$-primitive point $P\in E(\beta)$,  $\msep$ is zero-dimensional (and not just of virtual dimension zero). By \cite[Proposition 4.9]{Beh96} its virtual fundamental class equals the fundamental class
\[
[\msep]^{\mathrm{vir}}=[\msep].
\]
As a consequence (\cite[Corollary 4.10]{Beh96}) the calculation of $m_\beta^P=\nsep$ is reduced to finding $M_{\beta,P}$ and to computing the lengths of the log maps induced by $n : \pone \to C$ for $C\in M_{\beta,P}$. The contribution of such curves is the content of the next Proposition. 
For an (integral) rational curve $C$, recall that the \emph{compactified Jacobian} $\ol{\pic}^0(C)$ of $C$ is the scheme parametrizing torsion free, rank 1, degree 0 sheaves on $C$.

\begin{Proposition}[See Proposition \ref{propocontro}]
\label{prop:logk3contr}
Let $(S,E)$ be a regular log Calabi-Yau surface with smooth divisor and let $\beta\in\hhh_2(S,\ZZ)$ be a curve class. Let $C$ be an (integral) rational curve of class $\beta$ maximally tangent to $E$ at $P$, and denote the normalization map by $n:\pone\to C$. 
Then:
\begin{enumerate}
\item The map $n$ gives an isolated point of $\msep$, 
and contributes a positive integer to $\nsep$. 
\item If $C$ is immersed outside $P$, $[n]$ contributes $1$ to $\nsep$.
\item[(3)] If $C$ is smooth at $P$, then it contributes $\multc$ to $\nsep$.
\item[(4)] If $C$ is smooth at $P$ and has arithmetic genus $1$, then its contribution $\multc$ equals $e(C)$ and can be recovered from the associated elliptic fibration.
\end{enumerate}
Assume that $S$ is del Pezzo.
\begin{enumerate}
\item[(5)] Let $P\in E(\beta)$ be $\beta$-primitive. Assume that all $C\in M_{\beta,P}$ are smooth at $P$. Then
\[
m^P_\beta  = \sum_{C\in M_{\beta,P}} \multc.
\]
\end{enumerate}
\end{Proposition}

In the case of $\ptwo$, Conjecture \ref{conj3} 
was originally stated for $\ptwo$ in \cite[Remark 2.2]{tak log ms}. Using calculations from \cite{tak compl}, it was verified that it holds for $\mathbb P^2$ and degree $\leq 6$, as well as for degree 7 and 8 under certain technical hypotheses.
The combination of Conjecture \ref{conj3} and Proposition \ref{prop:logk3contr} yields an enumerative interpretation of the $n_\beta$. For an explanation of the factor $w$, see \S\ref{sec:open-closed}.

Recall that $p_a(\beta):=\frac{1}{2} \beta (\beta+K_S) + 1$ is the arithmetic genus of $\beta$.

\begin{remark}
Some comments on the assumptions in Proposition \ref{prop:logk3contr} are in order. 
Let $L\in\pic^\beta(S)$. Recall that $L$ is said to be $\delta$-very ample if $\hhh^0(L) \to \hhh^0(L\big{|}_Z )$ is onto for all length-$(\delta + 1)$ subschemes $Z\subset S$. It is known \cite[Proposition 2.1]{kst} that if $L$ is $p_a(\beta)$-very ample then a general $p_a(\beta)$-dimensional linear system $\mathbb{P} \subset |L|$ contains a finite number of rational curves and they all are $p_a(\beta)$-nodal. Hence any such curve $C$ is immersed and $\multc=1$.

By Proposition \ref{prop:rel sev}, being maximally tangent at $P\in E(\beta)$ picks out a $p_a(\beta)$-dimensional linear system $|\shO_S(\beta,P)| \subset |L|$. We however do not know whether $p_a(\beta)$-very ampleness combined with $E$ general would guarantee that the finite number of rational curves in $|\shO_S(\beta,P)|$ are immersed away from $P$. As an illustration, assume $E\subset\ptwo$ is given by $Y^2Z-X^3-Z^3$. Then the cuspidal cubic $Y^2Z-X^3$ is maximally tangent to $E$ at $P=[0:1:0]$ and contributes 2 to $\shN^P_\beta(\ptwo,E)$. If however $E$ is given by $Y^2Z-X^3-aXZ^2-bZ^3$, for $a\neq0$, then there are two nodal cubics meeting $E$ at $P$ each contributing 1 to $\shN^P_\beta(\ptwo,E)$. This is one of the cases of Proposition \ref{prop:lala}.

As for the requirement that each $C\in M_{\beta,P}$ is smooth at $P$, we expect this to be true for general $(S, E)$, but currently do not know of a proof. In the case of cubics in $\ptwo$, this is ruled out in Remark \ref{rem:52}.

It is known that every rational curve in a primitive class on a general K3 surface 
is nodal (\cite{XCh}), 
and it may be reasonable to expect that any 
$C\in M_{\beta, P}$ is nodal, and hence smooth at $P$, for a general $(S, E)$ 
and $\beta$-primitive $P$. 
\end{remark}

\begin{definition}\label{def:lineconicclasses}
Recall that \emph{line classes} on $S$ are the classes $l\in \mathrm{Pic}(S)$ such that $l^2=-1$ and $-K_S.l=1$. \emph{Conic classes} are the classes $D\in \mathrm{Pic}(S)$ such that $p_a(D)=0$ and $-K_S.D=2$. Our main result is as follows:
\end{definition}

\begin{thm}\label{thmthm}
Let $S$ be a del Pezzo surface, $E$ a smooth anticanonical divisor and $\beta$ a curve class on $S$. Assume that $\beta$ is a line class, a conic class or a nef and big class. Then Conjecture \ref{conj3} holds if $p_a(\beta)=0$ or $1$.
Assuming that $(S, E)$ is general, it holds for classes $\beta$ of arithmetic genus 2 as well.

\end{thm}

The proof of Theorem~\ref{thmthm} proceeds by calculation of both of the relevant log and local BPS numbers. On the log side we find in Section \ref{sec:calc} all rational curves of a given class in $S$ that are maximally tangent to $E$ at a primitive point. Using Proposition \ref{prop:logk3contr}, we calculate the corresponding log BPS numbers. Our calculations are summarized in Theorem \ref{thm:logcalc}. On the local side, we carried out the calculation in \cite{CGKT} via wall-crossing with stable pairs:

\begin{thm}[Theorem 1.1 in \cite{CGKT}] \label{thm:localcalc}
Assume that $S$ is a del Pezzo surface, that $E$ is a smooth anticanonical divisor and that $\beta$ is a line class, a conic class or a nef and big curve class. Let $\eta$ be the number of line classes $l$ such that $\beta. l =0$. We denote by $S_8$ the del Pezzo surface obtained by blowing up $\ptwo$ in 8 general points.
\begin{enumerate}
\item If $p_a(\beta)=0$, then $n_\beta= (-1)^{w-1} w$.
\item If $p_a(\beta)=1$ and $\beta\ne -K_{S_8}$, then $n_\beta= (-1)^{w-1} w (e(S)-\eta)$.
\item If $\beta=-K_{S_8}$, then $n_\beta=12$. 
\item If $p_a(\beta)=2$ and $\beta\neq-2K_{S_8}$, then $n_\beta= (-1)^{w-1} w \left(\binom{e(S)-\eta}{2}+5\right)$.
\end{enumerate}
\end{thm}

\begin{remark}\label{rem}
Note that for $\beta = -2K_{S_8}$, the log BPS number is calculated to be $66$ in \cite{tak compl} as reviewed in Section \ref{sec:g2high}. This matches with the corresponding (physics) computation carried out in \cite{HKP}. Note that in \cite[Section 5.1]{HKP}, the refined BPS index is calculated for the combined classes with $w=2$, which suffices for the verification since we know the local and log BPS numbers for all the other $w=2$ classes.

\end{remark}

\begin{remark}
Note that in the situation of Theorem~\ref{thm:localcalc}, the numbers $(-1)^{w-1}n_\beta/w$ only depend on the arithmetic genus of $\beta$ and on topological numbers of $S$. This structure bears similarities to the works \cite{Gott,kst,KT,LiTzeng,Tzeng,Tzengtang} on Severi degrees and extensions thereof. One might expect there to be some universal polynomials that calculate the $(-1)^{w-1}n_\beta/w$.
\end{remark}

Assuming Conjecture \ref{conj3}, we have $w | n_\beta$. 
This is also a consequence of the following conjecture on the cohomology of $\mathcal{M}_\beta$. For line, conic or nef and big classes of arithmetic genus up to 2 it is proven in \cite{CGKT} 
except in the case $\beta=-2K_{S_8}$.

\begin{Conjecture}[Conjecture 1.2 in \cite{CGKT}]\label{conj:div1}
The Poincar\'{e} polynomial $P_t(\mathcal{M}_\beta)$ has a factor of $P_t(\PP^{w-1})$. Consequently, $n_\beta$ is divisible by $w$. 
\end{Conjecture}

In a somewhat orthogonal direction, note Conjecture 44 of \cite{Bou2}, which stipulates a relationship, after a change of variable, of $P_t(\mathcal{M}_\beta)$ with a generating function of certain higher genus log Gromov-Witten invariants. Combining the two suggests a reconstruction result of higher genus log GW invariants in terms of genus 0 invariants.

There is a further unexpected connection that arises from of our calculations. 
We believe that this is a manifestation of a more general phenomenon:

\begin{Proposition}[See Proposition \ref{prop:loop}] \label{prop:loopy}
Let $(S,E)$ be a log Calabi-Yau surface pair.
Let $C$ be an integral nodal rational curve in $S$ that meets $E$ with maximal tangency at $P$. Let $l\in\NN$.
Denote by $\contr^{\bps}(l,C)$ the contribution of multiple covers over $C$ to the log BPS invariant $m^P_{l[C]}$ (Definition \ref{def:loopy}). Then
\[
\contr^{\bps}(l,C) = \dt^{(C.E-1)}_l,
\]
the $l$th generalized Donaldson-Thomas invariant of the $(C.E-1)$-loop quiver.
\end{Proposition}

\subsection{Open-closed duality}\label{sec:open-closed}

We describe a heuristic that views Conjecture \ref{conj3} as an instance of open-closed duality \cite{phys1,phys2,phys3}, i.e.\ a correspondence between disk and curve counts. Note that this correspondence is expected to hold as the BPS version of the (algebro-geometric) proof of the main result of \cite{GGR}.

Assume that $S$ is del Pezzo and $E$ anticanonical. Let $P\in E(\beta)$ be $\beta$-primitive and let $C\in M_{\beta,P}$. We first explain how $C$ should be viewed as a holomorphic disk $\Sigma \to S\setminus E$ with boundary mapping to a special Lagrangian $L$ near $E$. After that, we describe a procedure that associates to this disk (modulo multiplicity) a rational curve in $\tot(K_S)$.

The complement $S\setminus E$ is symplectomorphic to $V\setminus \partial V$, where $V$ is a Liouville manifold with boundary $\partial V$ a circle bundle in $N_{E/S}$. Moreover, a symplectic tubular neighborhood  $U$ of $E$ in $S$ is an open disk bundle in $N_{E/S}$. In particular, $\partial\ol{U}\to E$ is a circle bundle. Gluing $V$ and $\ol{U}$ along their boundaries yields the symplectic sum $M$.

According to \cite[Lemma 2.7]{DL}, $C$ corresponds to a holomorphic map from an open disk to $V\setminus \partial V$ with boundary converging to a $w$-multiple orbit of the fiber of $\partial V$ over $P$.\footnote{Note that \cite{DL} imposes restrictive conditions on the pairs considered, in order to guarantee transversality of the relevant moduli spaces. In our case, the log BPS invariants are weighted counts of curves, so no virtual techniques are required and the proof of \cite[Lemma 2.7]{DL} carries through.}
According to \cite[Conjecture 7.3]{Au}, special Lagrangian SYZ-fibers $L$ of $\ol{U}\cap (S\setminus E)$ are (a perturbation of) circle bundles $L \to L_E$, with fibers in $N_{S/E}$ and where $L_E$ is a special Lagrangian in $E$.

Combining this, one expects there to be a special Lagrangian $S\setminus E \supset \partial\ol{U} \supset L \to L_E\subset E$ such that $C$ correspond to a closed disk $\Sigma \to S\setminus E$ with boundary $\partial \Sigma$ mapping (in degree $w$) to a vanishing cycle of $L$. Note that after perturbation, it is expected that the same $L$ works for all $C\in M_{\beta,P}$.

Next, we join the image of $\Sigma$ with a $w$-multiple of the fiber of $\ol{U}\to E$ over $P$ to obtain a holomorphic map $\pone\to M$.  The choice of multiple cover introduces a factor of $(-1)^{w-1}/w^2$, cf.\ \cite{GGR}. We add a suitable line bundle on $M$ (see \cite{GGR} for details) such that (1) it does not affect the count of disks in $S\setminus E$ and (2) the total space of the line bundle deforms to $\tot(K_S)$. 

In this deformation, $\pone \to M$ deforms to a rational curve $\widetilde{C}$ in $\tot(K_S)$ which is counted $(-1)^{w-1}/w^2$ times. Furthermore, the deformation formula, cf.\ \cite{GGR}, introduces a factor $w$ in that $w$ local curves are expected to degenerate to the same log curve. Finally we need to take into account that $|E(\beta)|=w^2$ (assuming as in Conjecture \ref{conj1} that each point contributes the same). Putting this together yields the expectation that
\[
n_\beta = |E(\beta)| \cdot w \cdot \frac{(-1)^{w-1}}{w^2} \, m_\beta^{tot},
\]
which simplifies to Conjecture \ref{conj3}.
Finally, note that the contribution of $C$ to $m^P_{[C]}$ and that of $\widetilde{C}$ to $n_{[\widetilde{C}]}$ are both expected to be  given by the Euler characteristic of the compactified Jacobian.

\subsection{Pairs of maximal boundary}   Starting with a smooth rational surface $S$, there are two fundamentally different ways of obtaining a log Calabi-Yau surface. In the present paper, we proceed by choosing a smooth anticanonical divisor. This can be viewed as the tail of a type II semistable degeneration of a K3 surface. We conjecture that the resulting invariants are related to the local geometry of $S$, the total space $\tot(K_S)$ of the canonical bundle on $S$.

The other way proceeds by choosing a singular normal crossings divisor. This means that the divisor is either a nodal genus 1 curve or a cycle of rational curves. This case is usually referred to as being of \emph{maximal boundary}, which means that there is a toroidal structure near the divisor, and is treated in \cite{GPS10,GHK,Bou3}. One example thereof concerns components of maximally unipotent type III degenerations of K3 surfaces. 
This is described in the introduction of \cite{GS06} and the full mirror symmetry picture will be detailed in the upcoming work \cite{GHKS}.

Considering refined contributions in the maximal boundary toric case lead Bousseau in \cite{Bou} to a remarkable result stating that generating series of higher genus log GW invariants of toric surfaces are, after a suitable change of variables, refined Block-G\"ottsche counts of tropical curves.
The result of \cite{Bou} and our conjectures have the same origin. Namely, that the log GW/log BPS numbers very much behave like an algebraic version of open GW invariants, see \cite{GS16}.

\subsection{Analogies with K3 surfaces}

Log BPS numbers feature some striking parallels with counts of rational curves in K3 surfaces. On the one hand, denote by $(S,E)$ a log Calabi-Yau surface pair with $S$ del Pezzo. Let $\beta\in\hhh_2(S,\ZZ)$ be a curve class and let $P\in E(\beta)$ be $\beta$-primitive. Consider the $p_a(\beta)$-dimensional linear system $|\shO_S(\beta,P)|$ (Definition \ref{def:linsys}) of curves of class $\beta$ with maximal intersection multiplicity with $E$ at $P$. On the other hand, denote by $X$ a K3 surface, let $g\geq 0$ and consider a $g$-dimensional complete linear system $L$ of curves of genus $g$ on $X$.

Firstly, by Proposition \ref{prop:logk3contr}, a rational curve $C\in |\shO_S(\beta,P)|$, smooth at $P$, contributes $\multc$ to $m_\beta^P$. In parallel, the contribution of a rational curve $C\in L$ to the BPS count of the Yau-Zaslow formula \cite{YZ} is given (see \cite{B,FGS}) by $\multc$. 

Secondly, choosing $P$ to be $\beta$-primitive amounts to choosing a point on $E$ of maximal order (for a suitable choice of zero element, see Lemma \ref{lem:order}). This is analogous to taking the reduced class for K3 surfaces. The moduli space of stable maps (of genus 0 and given class) in $X$ is of virtual dimension $-1$ and hence all Gromov-Witten invariants are zero. This is because there is a non-algebraic deformation of $X$ in which there are no rational curves. This issue is resolved by restricting the deformation space to algebraic deformations (the Noether-Lefschetz locus). Similarly, assume that $(S,E)$ occurs as the tail of a type II (algebraic) degeneration of $X$. To simplify, assume that the degeneration consists of another component $(S',E')$ and that $S$ and $S'$ are glued together along $E\simeq E'$. Denote by $N$ and $N'$ the respective normal bundles of the components of the central fiber. For the deformation to be algebraic, $N|_E=N'|_E$ must hold. Consider a maximally tangent rational curve $C$ in $(S,E)$, resp. $C'$ in $(S',E)$, each meeting $E$ in $P$ in tangency $w=[C].E=[C'].E$. In order to algebraically deform $C\cup_P C'$ to a rational curve in $X$, $P$ should have the same order with respect to $[C]$ and $[C']$, for example by requiring $P$ to be both $[C]$- and $[C']$-primitive.

The third analogy is that $m_\beta^P$ also seems to depend on intersection numbers only. The BPS number of rational curves in $|L|$ depends only on the intersection number $2g-2$ (conjectured in \cite{MP,KKV}, proven in \cite{KMPS} in genus 0, in \cite{MPT} in all genus and primitive classes and in \cite{PT} without the primitivity assumption). In our calculations (for $p_a(\beta)\leq2$), $m_\beta^P$ depends only on $e(S)-\eta$, where $\eta$ is the number of line classes $l$ such that $\beta.l=0$.

For K3 surfaces, \cite{Lin1,Lin2} prove an equivalence between counts of Maslov index 0 disks with boundary on a SYZ-fiber and tropical curves in the base. In the log setting of this paper, the analogous correspondence is expected to hold between the log BPS numbers and tropical curves in the scattering diagrams of \cite{CPS}.

We end this introduction by mentioning that (combinations of) the log BPS numbers introduced here are numbers that occur in other contexts such as in \cite{Stienstra}. Throughout this paper, we work over $\CC$.

\addtocontents{toc}{\protect\setcounter{tocdepth}{0}}
\section*{Acknowledgements}
\addtocontents{toc}{\protect\setcounter{tocdepth}{1}}

We wish to thank Pierrick Bousseau for enlightening conversations on multiple cover formulas of open Gromov-Witten invariants that led to Definition \ref{def1} and Conjecture \ref{conj1} of the present paper. We thank Ben Davison for discussing quiver DT extensions of Section \ref{sec:loop}. Many thanks are owed to Tom Graber for discussing several aspects of the paper that relate to counting maximally tangent curves. We are grateful for the conversations with Martijn Kool on sheaf-theoretic aspects of this paper, and for the ones with Yu-Shen Lin and Dhruv Ranganathan on the analogies with K3 surfaces. We are indebted to Helge Ruddat for the repeated patient explanations on log GW theory. We learned many of the ideas underlying this paper from Bernd Siebert and are thankful to Richard Thomas for instructive discussions on the deformation theory of the log BPS numbers and the multiplicity of the curves that occur. In addition, the authors would like to thank Murad Alim, Florian Beck, Hans-Christian Graf von Bothmer, Kiryong Chung, Samuel le Fourn, Mark Gross, Paul Hacking, Young-Hoon Kiem, Davesh Maulik, Sam Molcho, Rahul Pandharipande, Jan Stienstra and Jonathan Wise for enlightening conversations on several aspects relating to this work. Finally we thank the anonymous referee for excellent comments that improved the quality of this manuscript. JC is supported by the Korea NRF grant NRF-2018R1C1B6005600. MvG is supported by the German Research Foundation DFG-RTG-1670 and the European Commission Research Executive Agency MSCA-IF-746554. SK is supported in part by NSF grant DMS-1502170 and NSF grant DMS-1802242, as well as by NSF grant DMS-1440140 while in residence at MSRI in Spring, 2018. NT is supported by JSPS KAKENHI Grant Number JP17K05204. This project has received funding from the European Union's Horizon 2020 research and innovation
programme under the Marie Sklodowska-Curie grant agreement No 746554.

\section{Maximally tangent stable log maps}    \label{sec:logGW}

Let $(S,E)$ be a log Calabi-Yau surface with a smooth divisor, i.e.\ $S$ is a smooth projective surface and $E$ a smooth anticanonical divisor, in this case an elliptic curve or a disjoint union of elliptic curves.
Log BPS numbers are derived from counting rational curves in $S$ with a single point of maximal tangency along $E$. This can be done in several ways. One may consider relative Gromov-Witten invariants which are defined for smooth very ample divisors in \cite{Ga02} and for any smooth divisors in \cite{Li01,Li02}. In our setting this approach is taken in \cite{tak mult}. For our purpose, cf. the proof of Proposition \ref{prop:logk3contr}, the appropriate setting is the further generalization to log Gromov-Witten invariants \cite{GS13}, where instead of a divisor only a log structure on the target is prescribed. Note also the analogous construction \cite{AC,Ch}, as well as, for smooth divisors, the treatment in \cite{KL16}. All these invariants agree in the case at hand, cf. \cite{AMW}.

We describe the moduli space of \emph{basic stable log maps} of \cite{GS13}, in the setting of importance to us, namely in genus 0 with one condition of maximal tangency along $E$. 
Note that by Proposition \ref{prop:finite} below, the virtual dimension is zero and hence we need not consider insertions.

\subsection{Genus $0$ stable log maps of maximal tangency}\label{sec:description_log_maps}

The schemes in this section will be endowed with a log structure. We do not distinguish notationally when we consider their underlying schemes as it will be clear from the context. We write $x$ to denote the marked point or a node. When we wish to emphasize that $x$ is the marked point, resp.\ a node, we sometimes write $x_1$, resp.\ $q$.

Let $X$ be a smooth variety and $D$ a smooth divisor on $X$. 
We view $X$ as the log scheme $(X, \shM_X)$ 
given by the divisorial log structure 
$\shM_X=\shM_{(X, D)}$. 
Let $\beta\in\hhh_2(X,\ZZ)$ be a curve class.

\begin{definition}
Let $(C/W, \{x_1\})$ be a $1$-marked pre-stable log curve 
(\cite[Def. 1.3]{GS13}) 
over a log point $W=(\Spec \kappa, Q)$ 
where $\kappa$ is an algebraically closed field over $\CC$, 
and $(C/W, \{x_1\}, f)$ a stable log map 
(i.e., $f: C\to X$ is a log morphism over $\Spec \CC$ 
and $f$ is a stable map, see \cite[Def. 1.6]{GS13}). 

It is called 
a stable log map of \emph{maximal tangency} 
of genus $0$ and class $\beta$ 
if the following hold: 

(i)
$C$ is of arithmetic genus $0$, $f_*[C]=\beta$. 

(ii)
the natural map 
\[
\Gamma(X, \overline{\shM}_X)\cong\NN\longrightarrow \overline{\shM}_{C, x_1}\cong Q\oplus\NN
\xrightarrow{\; pr_2 \;} \NN 
\]
is given by $1\mapsto D.\beta$. 

We will later see (Proposition \ref{prop_description_log_maps}) that (ii) follows from other conditions. 
\end{definition}
In the language of \cite[Def. 3.1]{GS13}, 
this is the case $g=0$, $k=1$, the condition $A$ provided by $\beta$, 
$Z_1=D$ 
and $s_1\in\Gamma({D}, (\overline{\shM}_D^{\textrm{gp}})^*)$ given by 
$\overline{\shM}_D^{\textrm{gp}}\simeq \ZZ_D \to \ZZ_D, \: 
1\mapsto D.\beta$. 

By \cite[Prop. 1.24]{GS13}, 
a stable log map $f$ as above is induced from a \emph{basic} stable log map 
over $(\Spec \kappa, Q^{\textrm{basic}})$. 
Since $Q^{\textrm{basic}}$ is a toric monoid, 
we can take a local homomorphism $Q^{\textrm{basic}}\to\NN$ 
and consider the induced stable log map 
over the standard log point $(\Spec \kappa, \NN)$. 
Hence, to study the underlying morphism of schemes, 
we may consider stable log maps over the standard log point. 

To a stable log map, one can associate its 
graph, type, ``tropical data'' and ``$\tau$-rays''. 
We explain this in the case of 
a $1$-marked genus $0$ stable log map $f$ over $(\Spec \kappa, \NN)$ 
to $({X}, {D})$. 
See \cite[\S\S1.4]{GS13} for details. 
Let us write $\beta={f}_*[C]$. 
Note that we do \emph{not} assume the maximal tangency condition (ii) here. 

Notation: 
If $\eta$ is the generic point of an irreducible component of $C$, 
then $C_\eta$ denotes this irreducible component, the closure $\bar\eta$ of $\eta$.

\begin{definition}(Graph) 
The \emph{dual graph } $\Gamma$ of ${C}$ consists of the following data.
The vertex set  $V(\Gamma)$ is the set of irreducible components of $C$.  The edge set $E(\Gamma)$ consists of one unbounded edge in addition to a number of bounded edges.  The unbounded edge is attached to the vertex corresponding to the irreducible component of $C$ containing the marked point. There is a bounded edge for each node, connecting the vertices (possibly the same) corresponding to the irreducible components containing each of the two local analytic branches at the node.
\end{definition}

The following lemma gives some clue 
about how ${f}$ meets ${D}$. 
\begin{Lemma}\label{lem_f_inverse_d}
For any generic point $\eta$ of $C$, ${f}^{-1}({D})\cap C_\eta$ is 
either $C_\eta$ or consists of nodes and marked points. 
\end{Lemma}
\begin{proof}
This follows from \cite[Remark 1.9]{GS13}, 
which says that the stalk of ${f}^{-1}\overline{\shM}_X$ 
jumps (i.e. the generization map is not an isomorphism)
only at nodes and marked points. 
\end{proof}

\begin{definition}(Tropical data) 
Since the genus is $0$, 
a node $q$ is the intersection of two components $C_{\eta_1}$ and $C_{\eta_2}$. 
We denote by $e_q\in \NN_{>0}$ the positive integer 
such that the following holds: $\overline{\shM}_{C, q}$ is embedded as 
$\langle (e_q, 0), (1, 1), (0, e_q)\rangle\subseteq \NN\oplus\NN$, 
where $(1, 1)$ corresponds to the image of 
$1\in\NN\cong\overline{\shM}_{(\Spec\kappa, \NN)}$ 
and the projections are identified with the generization maps to 
$\overline{\shM}_{C, \eta_i}\cong\NN$. 

Let $\varphi: f^{-1}\overline{\shM}_X \to \overline{\shM}_C$ 
be the morphism induced from $f$. 
We define 
\[
V_\eta:=\varphi_\eta(1) \in\NN_{>0}
\]
if ${f}(\eta)\in {D}$, 
where $({f}^{-1}\overline{\shM}_X)_\eta$ is 
identified with $\NN$; otherwise $V_\eta:=0$. 

The tuple $((V_\eta)_\eta, (e_q)_q)$ is called the \emph{tropical data}.
\end{definition}

\begin{definition}(Type) 
For a node as above, write $\{i, j\}=\{1, 2\}$ and let 
\[
u_{\eta_i, q}:=\frac{V_{\eta_j}-V_{\eta_i}}{e_q} \in\ZZ, 
\]
which is $0$ unless ${f}(q)\in {D}$. 
(Note that this notation avoids the issue of ordering as in \cite{GS13}. 
This is possible because a node is the intersection of two components 
in this case.) 

For the marked point $x_1$, we define
\[
u_{\eta, x_1}:=(\varphi_{x_1}(1) \mod \overline{\shM}_{(\Spec\kappa, \NN)})  \in\NN
\]
if ${f}(x_1)\in {D}$, where $({f}^{-1}\overline{\shM}_X)_{x_1}$ is 
identified with $\NN$ and $\overline{\shM}_{C, x_1}/\overline{\shM}_{(\Spec\kappa, \NN)}$ 
is identified with $\NN_{x_1}$; 
otherwise $u_{\eta, x_1}:=0$. 

The data $((u_{\eta, q}), u_{\eta, x_1})$ is called the \emph{type}. 
\end{definition}
As is seen from the definition, 
the type is determined by the tropical data, except for $u_{\eta, x_1}$. 
We will see in Proposition \ref{prop_description_log_maps} that $u_{\eta, x_1}$ must be $D.\beta$. 

To define the $\tau$-rays, we first describe $N_\eta:=\Gamma(C_\eta, {f}^{-1}\overline{\shM}_X^{\textrm{gp}})^*$. If ${f}(\eta)\not\in {D}$, then $N_\eta$ consists of one copy of $\ZZ$ for each element of $(f|_{C_\eta})^{-1}(D)$, which necessarily is either a node or the marked point.

If ${f}(\eta)\in {D}$, then $N_\eta$ consists of one copy of $\ZZ$. To state the balancing condition in convenient form, we will describe $N_\eta$ as the direct sum $\bigoplus \ZZ$ with one component for each node and marked point in $C_\eta$, quotiented out by the kernel of the map to $\ZZ$ given by $(a_i)\mapsto \sum a_i$. Choosing the basis given by basic vectors $(\delta_{ij})_j$, the kernel is generated by the differences of basic vectors. 

\begin{definition}($\tau$-rays)\label{def_tau}
For each $\eta$, consider
$N_\eta=\Gamma(C_\eta, {f}^{-1}\overline{\shM}_X^{\textrm{gp}})^*$. 
Let $\Sigma_\eta$ denote
the set of nodes and marked points of $C_\eta$, so that
\begin{eqnarray*}
& N_\eta \cong  \bigoplus_{x\in \Sigma_\eta\cap {f}^{-1}({D})} \ZZ
& \qquad \hbox{if ${f}(\eta)\not\in {D}$}, \\
& N_\eta \cong 
\left(\bigoplus_{x\in \Sigma_\eta} \ZZ\right)
/ H \cong \ZZ 
& \qquad \hbox{if ${f}(\eta)\in {D}$}, 
\end{eqnarray*}
where $H$ is the subgroup generated by the differences of basic vectors 
and the isomorphism to $\ZZ$ is given by $(a_x)\mapsto \sum a_x$. 
Write the class of $(a_x)$ by $[(a_x)]$. 

Then, 
if $f(\eta)\not\in D$, 
$\tau_\eta=(\tau_x)$ is given by $\tau_x:=-\mu_x(({f}|_{C_\eta})^*{D})$, 
the multiplicity at $x$; 
if $f(\eta)\in D$, 
$\tau_\eta := - \deg ({f}|_{C_\eta})^*{D}$. 
(These are the $\tau_\eta^X$ of \cite[\S\S 1.4]{GS13}.) 
\end{definition}

Now these data satisfy the \emph{balancing condition}:
\begin{Proposition}(\cite[Prop. 1.15]{GS13})\label{prop_balancing}
For each $\eta$, 
\[
\tau_\eta + [(u_{\eta, x})]=0
\]
holds in $N_\eta$. 
\end{Proposition}

In Proposition \ref{prop_description_log_maps} 
we will give a description of $1$-marked genus $0$ stable log maps 
to $({X}, {D})$. 
In the proof, the following ordering on the components of ${C}$ 
will be useful. 

\begin{definition}
(1)
Let $\eta_0\in V(\Gamma)$ correspond to the component 
on which the marked point lies. 
For $\eta_1, \eta_2\in V(\Gamma)$, 
we write $\eta_1\preceq \eta_2$ 
if $\eta_2$ is on the unique simple path 
connecting $\eta_1$ and $\eta_0$. 

Note that this is a partial ordering since $\Gamma$ is a tree, 
and that $\eta_0$ is the largest element. 

(2)
We write $\eta_1\leftarrowtail \eta_2$ 
if $\eta_1$ and $\eta_2$ are adjacent and $\eta_1\prec \eta_2$. 
(Then $\preceq$ is the partial ordering 
generated by $\leftarrowtail$.)
In this case, 
if $q$ is the edge connecting $\eta_1$ and $\eta_2$, 
we write $\eta_1\leftarrowtail q$ and $q\leftarrowtail \eta_2$. 
We also write $\eta_0\leftarrowtail x_1$. 

(3)
For any $\eta$, there is a unique edge $q$ with $\eta\leftarrowtail q$. 
We denote this by $q(\eta)$. 

(4)
For $\eta\in V(\Gamma)$, 
let 
$C_{\preceq \eta}=\bigcup_{\eta'\preceq \eta} C_{\eta'}$. 
\end{definition}

\begin{Proposition}\label{prop_description_log_maps}
Assume that ${D}.{f}_*[C_\eta]\geq 0$ holds 
for any $\eta$. 

(1)
For each $\eta$, the inverse image 
${f}^{-1}({D})\cap C_{\preceq \eta}$ 
of ${D}$ on $C_{\preceq \eta}$ 
is either empty, $\{q(\eta)\}$, or 
$\bigcup_{\eta'\in V(\Gamma')} C_{\eta'}$
for a subtree $\Gamma'$ containing $\eta$. 

Thus, if $f^{-1}(D)\cap C_\eta\not=\emptyset$, 
we can identify $N_\eta$ with $\ZZ$ 
and $\tau_\eta$ with $-D.f_*[C_\eta]$. 

(2)
For each $\eta$, 
we have 
$u_{\eta, q(\eta)}={D}.{f}_*[C_{\preceq \eta}]$. 

In particular, $f$ is of maximal tangency. 
\end{Proposition}
Note that, if $f(C)\not\subseteq D$, 
we can write $V_\eta$ in terms of $e_q$ and $D.f_*[C_\eta]$, 
which also leads to relations between these data. 

\begin{proof}
We prove the assertions by induction on the maximum length 
of simple paths from $\eta$ to minimal vertices. 

Note that, if the first half of (1) is proven for $\eta$, 
the second half is easy to see from the description of 
$N_\eta$ and $\tau_\eta$ in Definition \ref{def_tau}. 

If $\eta$ is itself minimal, 
it has exactly one special point, $q(\eta)$. 
By Lemma \ref{lem_f_inverse_d}, 
${f}^{-1}({D})$ is either 
empty, $\{q(\eta)\}$ or $C_\eta$, so (1) holds. 
Thus $N_\eta$ can be identified with $\ZZ$ or $\{0\}$, 
and by Proposition \ref{prop_balancing}, 
we have $u_{\eta, q(\eta)}=-\tau_\eta={D}.{f}_*[C_\eta]$, 
so (2) holds. 

Now assume that the assertions hold 
for each $\eta'$ with $\eta'\leftarrowtail\eta$. 
If  ${f}^{-1}({D})\cap C_{\preceq \eta'}$ 
is empty for all such $\eta'$, 
then ${f}^{-1}({D})\cap C_{\preceq \eta}$ 
is empty or $\{q(\eta)\}$ by Lemma \ref{lem_f_inverse_d}. 

If ${f}^{-1}({D})\cap C_{\preceq \eta'}$ 
is nonempty for some $\eta'$, 
then it contains $q:=q(\eta')=C_{\eta'}\cap C_\eta$ 
and we have 
\[
V_\eta = V_{\eta'} + u_{\eta', q}e_q. 
\]
If ${f}^{-1}({D})\cap C_{\preceq \eta'}$ 
consists of $q$, then $u_{\eta', q}\geq {D}.{f}_*[C_{\eta'}]>0$. 
Otherwise, it contains $C_{\eta'}$ and $V_{\eta'}>0$ holds. 
In either case, we have $V_\eta>0$ and $C_\eta$ is contained in 
the inverse image of ${D}$, and the assertion (1) follows. 

Then we have 
\begin{eqnarray*}
u_{\eta, q(\eta)} & = & -\tau_\eta - \sum_{\eta'\leftarrowtail\eta} u_{\eta, q(\eta')} \\ 
& = & {D}.{f}_*[C_\eta] +  \sum_{\eta'\leftarrowtail\eta} u_{\eta', q(\eta')} \\
& = & {D}.{f}_*[C_\eta] +  
\sum_{\eta'\leftarrowtail\eta} {D}.{f}_*[C_{\preceq \eta'}], 
\end{eqnarray*}
so (2) follows. 
\end{proof}

\begin{Corollary}\label{cor_description_log_maps}
Let $X$ be a divisorial log scheme 
given by a smooth variety ${X}$ 
and a smooth divisor ${D}$. 

For a genus 0 stable log map $f: ({C}, x_1)\to X$, 
assume the following: 
\begin{itemize}
\item
$w:={D}.{f}_*[{C}]>0$ 
and $w_i:={D}.{f}_*[C_i]\geq 0$ 
for any irreducible component $C_i$ of ${C}$. 
\item
If $C_i$ is an irreducible component of ${C}$ 
that is not collapsed by ${f}$, 
then ${f}(C_i)\not\subseteq{D}$. 
\end{itemize}
Then it is of maximal tangency, 
and the following holds. 
\begin{enumerate}
\item
${f}({C})\cap {D}$ consists of one point $P$. 
\item
If there is only $1$ non-collapsed component, 
then ${C}\cong \PP^1$ and $f^*(D)=w x_1$. 
\item
If there are at least $2$ non-collapsed components, 
and ${D}.{f}_*[C_i]> 0$ holds for non-collapsed components, 
then ${C}$ is given by adding $C_i=\PP^1$ as leaves 
to a tree ${C}'$ of $\PP^1$ collapsed to $P$, 
with maps $f_i: C_i\to S$ 
satisfying $f_i^*({D})=w_i(C_i\cap {C}')$. 
\end{enumerate}
\end{Corollary}

\begin{proof}
(1) 
By the assumption, ${f}({C})$ meets ${D}$. 
Applying Proposition \ref{prop_description_log_maps}, 
the inverse image of ${D}$ 
is the marked point or a tree of irreducible components, 
which are collapsed by the assumption. 
Hence its image consists of one point. 

(2)
A minimal vertex $\eta$ has only $1$ special point 
(a node or the marked point), 
so it is non-collapsed by stability. 
By assumption we have only one minimal vertex, 
so the graph is a chain. 
By stability it has only one vertex and the assertion follows. 

(3) 
Again each minimal vertex $\eta$ is non-collapsed, 
and it meets ${f}^{-1}({D})$ by assumption. 
Then we see from Proposition \ref{prop_description_log_maps} 
that $f^{-1}(D)$ is the union of all non-minimal components. 
By (1) they are mapped to a point. 
\end{proof}

\begin{center}
\begin{figure}\caption{ }\label{fig}
\bigskip
\begin{tikzpicture}[smooth,scale=0.7]
\coordinate (A) at (0.5,3);
\coordinate (B) at (0.5,-3);
\coordinate (C) at (-0.5,4.5);
\coordinate (D) at (0.5,0.25);
\coordinate (E) at (-0.5,-4.5);
\coordinate (F) at (0.5,-0.25);
\coordinate (X) at (-0.1,0);
\coordinate (1) at (-3,4);
\coordinate (2) at (0,4);
\coordinate (3) at (-3,3);
\coordinate (4) at (0,3);
\coordinate (5) at (-3,2);
\coordinate (6) at (0,2);
\coordinate (7) at (-3,-2);
\coordinate (8) at (0,-2);
\coordinate (9) at (-3,-3);
\coordinate (10) at (0,-3);
\filldraw [black] (X) circle (2pt) ;
\draw (A) to[out=-110,in=110] (B);
\draw (C) to[out=-90,in=130] (D);
\draw (E) to[out=90,in=-130] (F);
\draw (1) to[out=-10,in=185] (2);
\draw (3) to[out=-10,in=185] (4);
\draw (5) to[out=-10,in=185] (6);
\draw (7) to[out=-10,in=185] (8);
\draw (9) to[out=-10,in=185] (10);
\node [right] at (X) {$x$};
\node [below right] at (A) {$C_8$};
\node [above] at (C) {$C_6$};
\node [above left] at (E) {$C_7$};
\node [above right] at (1) {$C_1$};
\node [above right] at (3) {$C_2$};
\node [above right] at (5) {$C_3$};
\node [above right] at (7) {$C_4$};
\node [above right] at (9) {$C_5$};
\draw [->] (1.5,0) -- (4,0);
\draw (11,0) ellipse (6.5cm and 5.5cm);
\draw (9.3,0) ellipse (0.9cm and 1.35cm);
\draw (9,0) ellipse (1.2cm and 1.8cm);
\draw (8.7,0) ellipse (1.5cm and 2.25cm);
\draw (8.4,0) ellipse (1.8cm and 2.7cm);
\draw (8,0) ellipse (2.2cm and 3.3cm);
\filldraw [black] (10.23,0) circle (2.5pt) ;
\node [right] at (10.23,0) {$P$};
\draw[domain=10.23071:14.5] plot (\x, {sqrt((\x-12)*(\x-12)*(\x-12)-2*(\x-12)+2)});
\draw[domain=10.23071:14.5] plot (\x, -{sqrt((\x-12)*(\x-12)*(\x-12)-2*(\x-12)+2)});
\end{tikzpicture}
\end{figure}
\end{center}

We return to the setting of a log Calabi-Yau surface $(S,E)$. Let $\beta\in\hhh_2(S,\ZZ)$ be a curve class. 
Denote by $\mse$ the moduli space of maximally tangent genus 0 basic stable log maps to $(S,E)$ of degree $\beta$. Cf.\ Corollary \ref{cor_description_log_maps}, generic elements of the various strata of $\mse$ are as in Figure \ref{fig}, where components $C_6,C_7$ and $C_8$ are collapsed.
We state a result of \cite{Wise} in this setting. 

\begin{Proposition}[Corollary 1.2 in \cite{Wise}]\label{prop:Wise}

The forgetful morphism from $\mse$ to the moduli space of 1-marked genus 0 stable maps $\ol{\mmm}_{0,1}(S,\beta)$ is finite.

\end{Proposition}

The conditions of genus 0, degree $\beta$ and one marked point mapping in maximal tangency result in a finite number of types and hence make up a \emph{combinatorially finite} class, cf. \cite[Definition 3.3]{GS13}.
Hence $\mse$ admits a perfect obstruction theory, which is of virtual dimension 0, and yields a virtual fundamental class, as well as corresponding log Gromov-Witten invariants 
\[
\nse := \int_{[\mse]^{\mathrm{vir}}} 1 \in \QQ.
\]

We remark that much of this section could have been conveniently expressed 
in the language of tropical curves. 
(For example, under the assumption of the preceding corollary, 
the ordering of vertices is compatible with the ordering 
given by $V_\eta$.) 

\subsection{Rational curves of maximal tangency}

We state a proposition from \cite{tak compl} and  provide a proof for convenience.

\begin{Proposition}[Proposition 1.1 in \cite{tak compl}]\label{prop:finite}
Let $D \subset X$ be a smooth hypersurface of a smooth $n$-dimensional projective variety $X$. Assume that $|K_X + D| \neq \emptyset$. Let $\beta\in\hhh_2(X,\ZZ)$ be a curve class. Consider the set $U$ consisting of the union of all rational curves $C\subset X$ of degree $\beta$ having only one point of intersection with $D$. Then $U$ is contained in a proper Zariski-closed subset of $X$.
\end{Proposition}

\begin{proof}

Assume that $U$ is not contained in a proper Zariski-closed subset of $X$. Then we can find a diagram

\[
\xymatrix{
Y := M \times \pone \ar[r]^(0.65)f \ar@/_/[d]_{p} & X\\
M \ar@/_/[u]_{M_i},
}
\]
such that
\begin{itemize}
\item $M$ is a smooth variety of dimension $n-1$,
\item $f$ is dominant,
\item $M_i$ are disjoint sections of the projection map $p$,
\item $f^*(D)=\sum a_i M_i$ for $a_i$ positive integers.
\end{itemize}
Indeed, start with the (noncompact) moduli space of morphisms $g:\pone \to X$ which are birational onto its image $g(\pone)$ and such that the image curve meets $D$ in one point and consider a $(n-1)$-dimensional subspace whose universal curve maps dominantly to $X$. Note that $f^{-1}(D)$ is not necessarily a section of $p$:
We just assume that $f\left(f^{-1}(D)\cap (\{Q\}\times \PP^1)\right)$ is
one point for any $Q\in M$, not that $f^{-1}(D)\cap (\{Q\}\times \PP^1)$
is one point, although the latter is sufficient in what follows. After possibly taking a quasi-finite cover, this is the space $M$ with $f^*D=\sum a_i M_i$, for $a_i$ positive integers and $M_i\simeq M$ disjoint sections of $p$.

By assumption, there is a non-zero $n$-form $\omega$ on $X$, which is regular away from $D$ and has at most logarithmic poles along $D$. The pullback $f^*\omega$ is a non-zero $n$-form on $M\times\pone$ and has at most a logarithmic pole along $\cup M_i$. Hence
$\residue_{M_i}f^*\omega=a_i (f|_{M_i})^* \residue_D\omega$. After possibly shrinking $M$, we take a non-vanishing holomorphic $(n-1)$-form $\omega'$ on $M$ and write
\[
f^*\omega = \omega''\wedge p^*\omega',
\]
for some relative $1$-form $\omega''$. Over a general fiber of $p$, $\omega''$ restricts to a non-zero $1$-form on $\pone$ with at most logarithmic poles along which residues are positive integers times a certain complex number, which is a contradiction.
\end{proof}

\subsection{Point-dependence}\label{sec:pt-dep}

Assume in this section that $S$ is a regular surface (which we defined to mean $\hh^1(\shO_S)=0$) 
and $E$ is an elliptic curve on $S$. 
Let $\beta\in\hhh_2(S,\ZZ)$ be a curve class and recall that $w=\beta.E$, 
and assume that $w>0$. 
Then there is a unique $L\in\pic(S)$ such that $c_1(L)$ is Poincar\'e dual to $\beta$.
We use the notation $\beta|_E:= L|_E \in\pic^w(E)$ for the induced class. 
Denote by $\pic^0(E)[w]$ the $w$-torsion points of $\pic^0(E)$.
For $P\in E$, we write $\beta|_E \sim wP$ to indicate that $\beta|_E = [wP]$ in $\pic^w(E)$.
\begin{definition}\label{def:ebeta}
\[
E(\beta) := \left\{P\in E \; \big{|} \;  \beta|_E \sim wP \right\}.
\]
\end{definition}
\begin{Lemma}\label{lem:ebetagroup}
$E(\beta)$ is a torsor for $\pic^0(E)[w]\simeq\ZZ/w\times\ZZ/w$.
\end{Lemma}

\begin{proof} 
Since $E\simeq \pic^1(E)$ is a torsor for $\pic^0(E)$, $E$ admits an effective action of $\pic^0(E)[w]$. In addition, the latter acts on $E(\beta)$ transitively, since for $P_1,P_2\in E(\beta)$, $wP_1 - w P_2 \sim 0$ so $P_1 - P_2 \in \pic^0(E)[w]$. Hence $E(\beta)$ is a torsor for $\pic^0(E)[w]$.
\end{proof}

We consider the images of log stable maps. 
\begin{definition}\label{def:mbeta}
Let $M_\beta$ be the set of 
image cycles of genus $0$ stable log maps to $(S, E)$ of class $\beta$ 
of maximal tangency. 

For $P\in E$, let $M_{\beta, P}:=\{D\in M_\beta \mid \mathrm{Supp}(D)\cap E=P\}$. 
\end{definition}

We also define a notion of maximal tangency for curves. 
\begin{definition}
If $(S, E)$ is a pair consisting of a variety and a divisor, 
a curve $D$ on $S$ is said to be maximally tangent to $E$ 
if $D$ meets $E$ at only $1$ point 
and has only $1$ branch there; 
or equivalently, if the inverse image of $E$ on the normalization of $D$ 
consists of $1$ point. 
\end{definition}
An irreducible proper rational curve $D$ maximally tangent to $E$ 
can also be considered as an $\AA^1$-curve on $S\setminus E$. 
The following proposition shows how elements of $M_\beta$ are 
related to such curves. 

\begin{Proposition}\label{prop_description_log_maps_delpezzo}
Let $S$ be a smooth surface, 
$E$ an elliptic curve on $S$ 
and $\beta\in\hhh_2(S,\ZZ)$ a curve class. 
Assume that $w:=E.\beta>0$. 

Consider a $1$-marked genus $0$ stable log map to $(S, E)$ of class $\beta$. 
Then it is of maximal tangency, 
and the underlying stable map $f : (C, x) \to S$ 
satisfies the following. 
\begin{enumerate}
\item
$f(C)\cap E$ consists of one point $P$. 
If $S$ is regular then $P\in E(\beta)$ 
and consequently $M_\beta = \coprod_{P\in E(\beta)} M_{\beta, P}$.
\item
If the image cycle is integral, 
then $C\cong \PP^1$ and $f^*(E)=w x$. 
In this case, $C':=f(C)$ is a rational curve maximally tangent to $E$, 
and $f$ is the normalization map. 
Conversely, if $C'$ is a rational curve maximally tangent to $E$, 
then the normalization map of $C'$ 
lifts to a (unique) genus $0$ stable log map of maximal tangency 
with image cycle $C'$. 
\item
If $E$ is ample, 
any element of $M_\beta$ is a sum of rational curves maximally tangent to $E$, 
meeting $E$ at the same point. 
\end{enumerate}
\end{Proposition}
\begin{proof}
Since $E$ is an elliptic curve and any component of $C$ is rational, 
Corollary \ref{cor_description_log_maps} applies. 
Thus it is of maximal tangency, 
and the image meets $E$ at one point, 
which necessarily belongs to $E(\beta)$ if $S$ is regular.
This proves (1). 

(2)
If the image cycle is integral, 
then $C$ has only $1$ non-collapsed component mapped birationally. 
Thus the first half follows from Corollary \ref{cor_description_log_maps} (2). 

The latter half can be easily checked. 

(3) follows from Corollary \ref{cor_description_log_maps} (3). 
\end{proof}

Because $E(\beta)$ is finite, we have the decomposition
\[
\mse  = \bigsqcup_{P\in E(\beta)} \msep.
\]
Moreover, as the obstruction theory of a disjoint union is the sum of the obstruction theories of each component, we obtain the finite decomposition
\begin{equation}\label{ndecomp}
\nse = \sum_{P\in E(\beta)} \nsep,
\end{equation}
where
\[
\nsep:= \int_{[\msep]^{\mathrm{vir}}} 1 \in \QQ
\]
are the genus $0$ log GW invariants of $(S,E)$ of degree $\beta$ maximally tangent to $E$ at $P$.

So far we have worked mainly with regular surfaces and elliptic curves on it. 
In the following proposition we illustrate what our surfaces are like. 
It appears that interesting things happen mostly on rational surfaces 
in the context of log Gromov-Witten theory 
of genus $0$ maximally tangent curves. 

\begin{Proposition}\label{prop:rat}
Let $S$ be a regular surface. 

(1)
If there exists a non-zero effective divisor $E$ such that $K_S+E\sim 0$, 
then $S$ is rational. 

(2) 
Let $E$ be a curve on $S$ 
and $\beta$ a curve class 
with $E.\beta>0$ and $\vdim\mse\geq 0$. 
\begin{itemize}
\item[(a)]
If $\beta$ is nef, then $S$ is rational, 
\item[(b)]
If $S$ is not rational and $\beta$ contains an irreducible 
(not necessarily reduced) member, 
then $|\beta|=\{mC\}$ where $C$ is a $(-1)$-curve with $E.C=1$. 
\end{itemize}

(3)
If $\beta$ is an ample class of arithmetic genus $1$ on $S$ 
containing an integral member, 
then $\beta=-K_S$ and hence $S$ is del Pezzo. 

In particular, if $S$ contains an ample elliptic curve $E$, 
then $S$ is del Pezzo and $E$ is anticanonical. 
\end{Proposition}

\begin{proof}
(1)
Since $|mK_S|=\emptyset$ for any $m>0$, 
$S$ is birationally ruled. 
From regularity it follows that $S$ is rational. 

(2)
The virtual dimension of the moduli space of stable maps $\ol{\mmm}_{0,0}(S,\beta)$ (i.e.\ before imposing the tangency condition) is $-K_S.\beta -1$. Imposing the condition of maximal tangency with $E$ cuts down the dimension by $E.\beta-1$. Therefore
\[
\vdim\mse=(-K_S-E).\beta.
\]
Thus $-K_S.\beta \geq E.\beta>0$ holds.

If $S$ is regular but not rational, by the Enriques-Kodaira classification, $S$ is birationally non-ruled. Let $\pi : S \to S'$ be the minimal model of $S$. Then $K_{S'}$ is nef and $K_S=\pi^*K_{S'} + \sum a_iF_i$, where the $F_i$ are the exceptional curves and $a_i>0$. 
Since $K_{S'}$ is nef, $(\sum a_iF_i).\beta\leq K_S.\beta<0$ and $\beta$ is not nef. 
If $\beta$ has an irreducible member, it is a multiple of one of $F_i$. 
From $K_S.\beta<0$ it follows that $K_S.F_i<0$, and $F_i$ is a $(-1)$-curve. 
Then $0<E.\beta\leq -K_S.\beta$ implies $E.F_i=1$. 

(3)
Let $C$ be an integral member of $|\beta|$. 
By regularity we have $h^1(\shO_S(K_S))=h^1(\shO_S)=0$, 
so there is an exact sequence
\[
0\to H^0(\shO_S(K_S))\to H^0(\shO_S(K_S+\beta)) \to H^0(\shO(K_S+\beta)|_C)\to 0. 
\]
Since $\shO(K_S+\beta)|_C\cong\omega_C\cong\shO_C$, 
it follows that there exists an element $D\in |K_S+\beta|$ such that $D\cap C=\emptyset$. 
Since $C$ is ample, $D$ must be $0$, hence $\beta\sim -K_S$. 
\end{proof}

\section{Local and log BPS numbers}

\label{sec:relGW}

Let $(S,E)$ be a log Calabi-Yau surface with smooth divisor. Whenever we talk about local GW invariants of $S$, we assume that $-K_S$ is ample.
We start by reviewing the (original) definition of local BPS invariants obtained by removing multiple cover contributions from the genus 0 local Gromov-Witten invariants of $S$ and yielding
integer invariants. The genus 0 local Gromov-Witten invariants of $S$ are the ordinary GW invariants of the noncompact Calabi-Yau threefold $X=\mathrm{Tot}(K_S)$. Whenever $-K_S$ is ample, these invariants can be defined via intersection theory on the compact moduli space of stable maps to $S$ as follows. Because $-K_S$ is ample, $\hhh^{0}(C,f^{*}K_{S})=0$ for a non-constant stable map $[f:C\to S]$. Therefore the line bundle $\shO (K_S)$ introduces no new deformations, just additional obstructions which we describe now.

For $m\geq 0$, denote by $\overline{\mmm}_{0,m}(S,\beta)$ the Deligne-Mumford moduli stacks of (isomorphism classes of) stable maps $[f:C\to S]$ of genus 0, with $m$ marked points and such that $f_{*}([C])=\beta$. There is a forgetful morphism
\[
\pi:\overline{\mmm}_{0,1}(S,\beta)\to\overline{\mmm}_{0,0}(S,\beta),
\]
which is the universal curve over $\overline{\mmm}_{0,0}(S,\beta)$. The latter carries a virtual fundamental class
\[
\left[\ol{\mmm}_{0,0}(S,\beta)\right]^{\mathrm{vir}}\in\hhh_{2\vdim}(\ol{\mmm}_{0,0}(S,\beta),\ZZ)
\]
of virtual dimension $\vdim=-K_S.\beta + (\dim S -3)=w-1$. The evaluation map
\[
ev:\overline{\mmm}_{0,1}(S,\beta)\to S
\]
determines the obstruction bundle
\[
\obstr:=R^{1}\pi_{*}ev^{*}K_{S},
\]
which is of rank $\vdim$ and has fiber $\hhh^{1}(C,f^{*}K_{S})$ over a stable map $[f:C\to S]$.

\begin{definition}
\emph{The genus $0$ degree $\beta$ local Gromov-Witten
invariant $\GW_\beta(X)$ of $S$} is defined as
\[
\GW_\beta(X):=\int_{[\overline{\mmm}_{0,0}(S,\beta)]^{\mathrm{vir}}}c_{\vdim}\left(\obstr\right)\in\mathbb{Q}.
\]

\end{definition}

We will make use of the following correspondence theorem. Its initial form for $\ptwo$ was conjectured in \cite{tak log ms} and proven in \cite{Ga03}. Then it was generalized in \cite{GGR} to any smooth projective variety of any dimension with maximal tangency condition along a smooth nef divisor (the most general statement is at the level of virtual fundamental classes). We state it in the generality needed for us.

\begin{thm}[See \cite{tak log ms,Ga03,GGR}] \label{thm:GGR} Assume that $E$ is ample. Then
\[
(-1)^{w-1}w \, \GW_\beta(X) = \mathcal N_\beta(S,E).
\]
\end{thm}

The relation between the $n_\beta$ and the $\GW_\beta(X)$ is given by the multiple cover formula (Aspinwall-Morrison formula):  
\begin{equation}\label{formmult}
\GW_\beta(X)=\sum_{k|\beta}\frac{1}{k^{3}}n_{\beta/k}.
\end{equation}
The Aspinwall-Morrison formula was proven for rigid rational curves in \cite{Manin, Voisin}.  
Note that for $S$ del Pezzo, \eqref{formmult} is equivalent to Definition \ref{defn:katz} as explained in Section 3.3 of \cite{CGKT}. 
Inverting this formula yields that
\[
n_\beta  =  \sum_{k|\beta} \frac{1}{k^3} \, \mu(k) \, \GW_{\beta/k}(X),
\]
where $\mu$ is the M\"obius function.
Under the assumption that $E$ is ample,
combining with Theorem \ref{thm:GGR}, we obtain that
\begin{equation}\label{eq:formnbeta}
(-1)^{w-1} w \, n_\beta   =  \sum_{k|\beta} \frac{1}{k^3} \, \mu(k) \, (-1)^{w-1} w \, (-1)^{\frac{w}{k}-1} \, \frac{1}{w/k} \, \shN_{\beta/k}(S,E).
\end{equation}
Noting that $w+w/k$ and $(k-1)w/k$ have the same parity, we rewrite (\ref{eq:formnbeta}) as follows.
\begin{equation}\label{eq:nbetarewrite}
(-1)^{w-1} w \, n_\beta   =  \sum_{k|\beta} \frac{(-1)^{(k-1)w/k}}{k^2} \mu(k) \,  \shN_{\beta/k}(S,E).
\end{equation}
By the decomposition \eqref{ndecomp}, this turns into 
\begin{equation}\label{form:locallog}
(-1)^{w-1} w \, n_\beta  = \sum_{k|\beta} \frac{(-1)^{(k-1)w/k}}{k^2} \mu(k)  \sum_{P\in E(\beta/k)} \mathcal N^P_{\beta/k}(S,E). 
\end{equation}
This suggests that we define the log BPS numbers at $P$ as
\begin{equation}\label{eq:explicitlogBPS}
m^P_\beta :=  \sum_{\left\{k|\beta \, : \, P\in E(\beta/k)  \right\}}   \frac{(-1)^{(k-1)w/k}}{k^2} \, \mu(k) \, \mathcal N^P_{\beta/k}(S,E).
\end{equation}
Note that $\shN^P_{\beta'}(S,E)=0$ if $P\not\in E(\beta')$. Setting $m^P_{\beta'}=0$ if $P\not\in E(\beta')$ and inverting equation \eqref{form:locallog} yields the following definition.

\begin{definition}[See Definition \ref{def1}]\label{multcover}
Let $(S,E)$ be a log Calabi-Yau surface with respect to $\beta\in\hhh_2(S,\ZZ)$ and assume that $S$ is regular and $E$ elliptic. We define $m^P_\beta$, the \emph{log BPS number at $P$}, via 
\[
\mc{N}^P_\beta(S,E) =  \sum_{k|\beta} \frac{(-1)^{(k-1)w/k}}{k^2} \, m^P_{\beta/k}.
\]
\end{definition}

Note that in the above sum, the number of terms varies with the arithmetic properties of $P$, 
as is illustrated by the examples of \S \ref{sec:evidencebig}. Note also that \cite[Remark 4.11]{tak mult} contains an equivalent description of this 
formula for $\IP^2$. 
By the above calculation,
\begin{equation}\label{eq:sum_log_in_local}
\sum_{P\in E(\beta)} m_\beta^P = (-1)^{w-1} w \, n_\beta.
\end{equation}

\begin{Conjecture}[See Conjecture \ref{conj1}]\label{conj:big}
Let $(S,E)$ be a log Calabi-Yau surface with respect to $\beta\in\hhh_2(S,\ZZ)$ and assume that $S$ is regular and $E$ elliptic. For all $P,P'\in E(\beta)$,
\[
m^P_\beta = m^{P'}_\beta.
\]
If in addition $E$ is ample, this is equivalent to the assertion that for all $P\in E(\beta)$,
\[
 n_\beta = (-1)^{w-1} \, w \, m^P_\beta.
\]
\end{Conjecture}

In \S\ref{sec:beta}, we elaborate on Conjecture \ref{conj:big} in the case of some special points $P\in E(\beta)$ which we call $\beta$-primitive (Definition \ref{def:prim}). For such $P$, $m^P_\beta$ is a weighted count of rational curves, see Proposition \ref{propocontro}. In that setting, the conjecture hence gives an enumerative interpretation of $n_\beta$.

\section{Primitive points of contact}\label{sec:logbps}

For this section and next, we assume that $S$ is a del Pezzo surface and consider certain points $P\in E(\beta)$ such that $\msep$ is zero-dimensional.

\subsection{Preliminaries on del Pezzo surfaces}

In this section, we collect basic facts about curve classes on del Pezzo surfaces. Let $S$ be a del Pezzo surface. Denote by $S_r$ the blowup of $\ptwo$ along $r$ general points. Then $S$ is either $S_r$ for $0\le r \le 8$ or $\pone\times\pone$. We will mainly consider the case $S=S_r$ and will make remarks for $\pone\times\pone$ separately whenever needed. The results of this paper hold for $\pone\times\pone$ as well. 

\begin{definition}
A class $\beta\in \hhh_2(S,\ZZ)$ is a \emph{curve class} if it can be represented by a nonempty subscheme of dimension one. We often consider $\beta$ as a divisor on $S$.
\end{definition}

Recall that $p_a(\beta):=\frac{1}{2} \beta (\beta+K_S) + 1$ is the arithmetic genus of $\beta$. Since del Pezzo surfaces are rational, by Poincar\'e duality, $\pic(S)\simeq\hhh_2(S,\ZZ)$. So when we write $|\shO_S(\beta)|$ or simply $|\beta|$, we mean the complete linear system $|L|$ for the unique $L\in\pic(S)$ such that $c_1(L)=\beta$.

For $S_r$, let $h$ be the pullback of $\mathcal{O}_{\ptwo} (1)$ and let $e_i$ for $1\le i\le r$ be the exceptional divisors. The Picard group $\mathrm{Pic}(S_r)$ is generated by $h$ and the $e_i$'s. The anticanonical divisor is $-K_{S_r}=3h-\sum_{i=1}^r e_i$. 
For $\pone\times\pone$, we denote by $h_1$ and $h_2$ the pullback of $\mathcal{O}_{\pone} (1)$ from each factor. The anticanonical divisor is $-K_{\pone\times\pone}=2h_1+2h_2$.

\begin{definition}
  A \emph{line class} on $S$ is a class $l\in \mathrm{Pic}(S)$ such that $l^2=-1$ and $(- K_S).l=1$. 
\end{definition}

It is well-known that each line class contains a unique integral line and there are only finitely many lines on $S$. 
\begin{example}
  By numerical calculation, we list all line classes up to permutation of the $e_i$'s: 
  \[e_i, (1; 1^2) ,(2; 1^5),(3; 2, 1^6), (4; 2^3, 1^5) , (5; 2^6, 1^2), (6; 3, 2^7). \]
  Here, we used the notation $(d; a_1,\cdots ,a_r)$ for the divisor $dh-\sum a_i e_i$. The superscripts indicate the number of repetitions.
\end{example}

\begin{definition}
In the case of $S_r$, a point $P\in E$ is said to be a \emph{flex point} if there is a curve $C$ of class $h$ such that $C$ meets $E$ only at $P$ (of tangency $h.E=3$).
\end{definition}

\begin{Lemma}[See \cite{Rocco}]\label{lem:Rocco}
Let $\beta\in\hhh_2(S,\ZZ)$ be a curve class containing an integral curve and with $p_a(\beta)\geq 1$. Assume that $\beta$ is not $-K_{S_7}, \, -K_{S_8} \text{ or } -2K_{S_8}$ (neither of which are very ample). Then $\beta$ is very ample if and only if there are no line classes $l$ such that $\beta.l=0$.
\end{Lemma}

\begin{proof}
We first note that $\beta$ is nef and big. 
In fact, let $C$ be an integral member of $\beta$. 
Since $S$ is del Pezzo, $K_S.C<0$ holds, 
so $C^2>0$ follows from $p_a(C)\geq 1$ and adjunction. 
Since $C$ is irreducible, it is nef. 

Now we apply the criterion of \cite{Rocco}. When $S=\PP^2$, the assertion holds.
When $S=\pone\times\pone$, the criterion is that $\beta.h_i\ge 1$ for $i=1,2$. By our assumptions, this is satisfied.

For $S=S_1$, the criterion is that $\beta.e_1\ge 1$ and that $\beta.(h-e_1)\ge 1$. The first condition follows from the assumption that there are no line classes $l$ with $\beta.l=0$ (and nefness). If we assume that the second condition does not hold, then by nefness of $\beta$ we have $\beta.(h-e_1)=0$. If we write $\beta=dh-ae_1$, this translates into $d=a$. This contradicts the bigness of $\beta$. Finally, when $S=S_r$ for $r \ge 2$, the criterion is that $\beta. l\ge 1$ for any line class $l$, which is satisfied by our assumptions.
\end{proof}

\subsection{$\beta$-primitive points} \label{sec:beta}

\begin{definition}\label{def:prim}
A point $P\in E$ is said to be $\beta$-primitive if $\beta|_{E}\sim wP$,
but there is no decomposition into non-zero pseudo-effective classes $\beta=\beta'+\beta''$,
with $\beta'|_{E}\sim w'P$, where $w'=\beta'.(-K_S)>0$.
\end{definition}

This definition guarantees that if $P$ is $\beta$-primitive, any curve of class $\beta$ meeting $E$ only at $P$ is integral and no multiple covers or maps with reducible image appear in the moduli space $\msep$.

\begin{remark}
For surfaces $S$ with Picard number $\rho(S)\geq 2$, the extremal rational curves have class $D$ with $p_a(D)=0$ and $D^2\leq0$. In the case of del Pezzo surfaces (other than $\ptwo$), the extremal rays are generated by line classes and the conic classes that don't decompose as sums of line classes. The latter are exactly the conic classes of the form $h-e_i$ when $S$ a blowup of $\ptwo$ and the fiber classes in the case of $\pone\times\pone$. The line and conic classes are classified in \cite[Examples 2.3 and 2.11]{CGKT}. In particular there are only finitely many extremal rays. So the effective cone is polyhedral, hence closed, and pseudo-effective is the same as effective for del Pezzo surfaces. In particular, there are only finitely many ways of decomposing $\beta$ as above. We leave the definition as is since it also applies to more general surfaces. Similarly, the results of this section may be adapted to hold more generally.
\end{remark}

\begin{Lemma}\label{lem:ptwo}
Assume that $S=\ptwo$ and let $d\geq 1$. Then the following are equivalent.
\begin{enumerate}
\item $P$ is $dh$-primitive.
\item $P$ is of order $3d$ for a choice (not necessarily all) of $0\in E$ a flex point.
\item For a fixed flex point $0\in E$,
\begin{equation*}
\begin{cases}
P \text{ is of order } 3d, & \text{ if } 3|d. \\
P \text{ is of order } d \text{ or } 3d, & \text{ if } 3 \ndiv d.
\end{cases}
\end{equation*}
\end{enumerate}
\end{Lemma}

\begin{proof}

Choose $0\in E$ to be any flex point. Let $\ord(P)$ be the order of $P$ for the resulting group law, so that $\ord(P)|3d$. From the definition, we see that $P$ is $\beta$-primitive if and only if $\ord(P)\ndiv 3d_1$ for all $d_1$ such that $1\leq d_1 < d$. This is equivalent to $(3)$. Condition $(2)$ then is obtained by noting that  all other flex points are of order $3$ with respect to $0$.
\end{proof}

\begin{Lemma}\label{lem:decomp}
Let $\beta\in\hhh_2(S,\ZZ)$ be a curve class. 
Assume that $(S,E)$ is general and let $P\in E(\beta)$. Then $P$ is not $\beta$-primitive if and only if there is a cycle class $\bar{\beta}$ and an integer $k>1$
such that $\beta=k\bar{\beta}$ and $\bar{\beta}|_{E}\sim\frac{w}{k}P$.
\end{Lemma}

\begin{proof}
We are looking at deformation classes of $(S,E)$. Since the Picard
group does not change, we consider a class $\gamma$ as being a class
in each deformation.

Consider a decomposition of $\beta$ into effective classes
$\beta=\beta'+\beta''$ with $\beta'|_E\sim w'P$.
Assume first that $\beta=k\beta'$ for $k\in\Q$. Note that we can check equalities with $\QQ$-coefficients as $\pic(S)$ is torsion-free. Choose $a,b\in\mathbb{Z}$ to satisfy $aw+bw'=gcd(w,w')$ and set
$\bar{\beta}=a\beta+b\beta'$. Then $\overline{\beta}|_{E}\sim gcd(w,w')P$. Since $w=kw'$, $akw'+bw'=gcd(w,w')$ and $\overline{\beta} = (ak+b)\beta' = \frac{gcd(w,w')}{w'}\beta'$. It follows that $\frac{w}{gcd(w,w')} \overline{\beta}=\frac{w}{w'}\beta'=k\beta'=\beta$. Since $\frac{w}{gcd(w,w')}$ is an integer greater than $1$, we are done with the case $\beta=k\beta'$.

Assume now that $\beta\neq k\beta'$ for any $k\in\Q$. This case excludes $\ptwo$. We will show that this can only happen on a proper closed subset of the moduli of the $(S,E)$. 
Let $S=S_{r}$ be the blowup of $\ptwo$ in $r$ general points $P_1,\dots,P_r$ for $1\leq r\leq8$. Then $E\subset S_r$ is the strict transform of a cubic in $\ptwo$ through $P_1,\dots,P_r$. We identify points on $E$ with those on this cubic.
Recall that $e_i$ is the exceptional divisor over $P_i$. 
Choose also a flex point $P_0$. 
Let
\[
\beta\sim dh-\sum a_{i}e_{i} \text{ and } \beta'\sim d'h-\sum a'_{i}e_{i}
\]
with $d,d',a_{i},a_i'\in\ZZ$. 
Note that there are only a finite number of possibilities for $\beta'$. Indeed, the set
\[
\left\{\beta'\in\hhh_2(S,\RR) \text{ effective and } (E.\beta')\leq (E.\beta) \right\}
\]
is compact (since the effective cone is closed).
We find that
\[
3d'P_{0}-\sum a'_{i}P_{i}\sim w'P \text{ on } E,
\]
as well as
\[
3dP_{0}-\sum a{}_{i}P_{i}\sim wP \text{ on } E,
\]
so that
\begin{equation}\label{eq:linequiv}
3(d'w-dw')P_{0}+\sum(w'a_{i}-wa'_{i})P_{i}\sim0.
\end{equation}
Assume that the $w'a_i-wa'_i$ are not all zero and view \eqref{eq:linequiv} as an equation in the $P_i$ on $E$. As such, it defines a proper closed subset of the set of tuples of blowup loci $\left\{(P_i)\right\}$. Outside of this set, 
\begin{equation}\label{eq:aident}
a_i=\frac{w}{w'}a'_i.
\end{equation}
By calculating the degree of the divisor (\ref{eq:linequiv}), we see that $d'w-dw'=0$, from which it follows that $\beta=\frac{d}{d'}\beta'$ using (\ref{eq:aident}). Thus $\beta\neq k\beta'$ is only possible on a closed proper subset of the parameter space.

In the case of $\pone\times\pone$, we blow up a point on $E$ away from $E(\beta)$ and deduce the result from the $S_2$ case.
\end{proof}

We introduce some more notation. Let $\beta$ be a curve class, let $(S,E)$ be general with respect to $\beta$ and let $P\in E(\beta)$. It follows from Lemma \ref{lem:decomp}, that there is a unique $k(\beta,P)\in \NN$ such that $P$ is $\beta/k(\beta,P)$-primitive. If moreover $\beta/k(\beta,P)$ is a primitive curve class, i.e.\ is a primitive vector in the cone of effective classes, we say that $P$ is a \emph{$\beta$-zero point} of $E$.

For example, flex points are $dh$-zero points for all $d\geq 1$. Indeed, in this case $k(dh,P)=d$ since $h$ is primitive (but no multiple of it is). For $S_r$ and a class $\beta = dh - \sum a_i e_i$ with $\gcd (d,a_1,\dots,a_r)=1$,  each point of $E(\beta)$ is both $\beta$-zero and $\beta$-primitive.

\begin{Lemma}\label{lem:order} Let $\beta$ be a curve class, let $(S,E)$ be general with respect to $\beta$ and let $P\in E(\beta)$. 
Then $P$ is $\beta$-primitive if and only if $P$ is of order $w$ for a choice (not necessarily all) of $\beta$-zero point $0\in E$.

\end{Lemma}

\begin{proof}
This is a natural extension of the proof of Lemma \ref{lem:ptwo}. 
\end{proof}

\begin{Proposition}\label{prop:existence} For each curve class $\beta$  and a general pair $(S,E)$, there is a $\beta$-primitive point $P\in E(\beta)$.
\end{Proposition}

\begin{proof}
This follows from Lemma \ref{lem:order} and the fact that $\ZZ/w\times\ZZ/w$ has elements of order $w$.
\end{proof}

\subsection{Log BPS numbers at primitive points}

\begin{Conjecture}[Special case of Conjecture \ref{conj:big}]\label{conj:well-def}
For a curve class $\beta$, for a general pair $(S,E)$ and for $\beta$-primitive $P,P'\in E(\beta)$,
\[
m^P_\beta = m^{P'}_\beta.
\]

\end{Conjecture}
The number $m^P_\beta=\nsep$ is invariant under log smooth deformation, cf. \cite{MR}. Deforming the triple $S\supset E\ni P$ keeping $E$ smooth and $P$ $\beta$-primitive is such a log smooth deformation. Hence a sufficient condition for Conjecture \ref{conj:well-def} to hold is that in the moduli space of such $S\supset E\ni P$ there is a unique connected component. 

For $P,P'\in E(\beta)$ that are $\ol{\beta}$-primitive for $\ol{\beta}=\beta/k$, $k\geq 1$, a similar argument is expected to show that $\nsep=\mc{N}^{P'}_\beta(S,E)$.

\smallskip
In fact, Conjecture~\ref{conj:well-def} is true for $S=\mathbb{P}^2$.

\begin{Lemma}\label{lem:p2irred}
  Let $S=\mathbb{P}^2$ and fix $d\ge 1$.  Then the universal locus
\[
U_d=\left\{(P,E)\mid E\ {\rm is\ a\ smooth\ plane\ cubic\ and\ }P{\rm \ is\ a\ }d\text{-}{\rm primitive\ point\ of\ }E\right\}
\]
of $d$-primitive points is irreducible.
\end{Lemma}

\begin{proof}
We adapt an argument from \cite{H}.  It suffices to show that for any smooth cubic $E$ and distinct $d$-primitive points $P,P'\in E$, we can find a path in $U_d$ connecting $(P,E)$ and $(P',E)$.

To that end, choose a flex $P_0$ of $E$ distinct from $P$ and $P'$.  As before, the map $Q\mapsto Q-P_0$ identifies $E(d)$ with the set of $3d$-torsion points of $\mathrm{Pic}^0(E)$.  The $d$-primitivity condition on a point $q\in E$ becomes
\begin{equation}
  \label{eq:primcond}
  Q-P_0\ {\rm is\ not\ }3d'\text{-}{\rm torsion\ for\ any\ }d'<d,
\end{equation}
independent of the choice of flex $P_0$.  In particular, $P-P_0$ and $P'-P_0$ satisfy (\ref{eq:primcond}).  Furthermore, by Lemma~\ref{lem:ptwo}, we know that the orders of $P-P_0$ and $P'-P_0$ are exactly $3d$ if $3|d$ and  either $d$ or $3d$  if $3\ndiv d$.

Up to change of coordinates in $\mathbb{P}^2$, we can identify $E$ with the image of $E_{\omega_1,\omega_2}:=\mathbb{C}/(\mathbb{Z}\omega_1\oplus\mathbb{Z}\omega_2)$ under the map 
\begin{equation}
  \label{eq:wpembed}
  \iota:E_{\omega_1,\omega_2}\to \mathbb{P}^2, \qquad \iota(z)=(\wp(z),\wp'(z),1) 
\end{equation}
for some independent periods $\omega_1$ and $\omega_2$, with $\iota(0)=P_0$.  In (\ref{eq:wpembed}), $\wp$ is the Weierstrass function on $\mathbb{C}$, periodic with respect to the lattice $\mathbb{Z}\omega_1\oplus\mathbb{Z}\omega_2$.  We put $\omega=(\omega_1,\omega_2)$ for convenience.  For some $m,n,m',n' \in ((1/3d)\ZZ)/\ZZ$, let $\widetilde{P}=m\omega_1+n\omega_2$ and $\widetilde{P}'=m'\omega_1+n'\omega_2$ be the $3d$-torsion points of $E_\omega$ corresponding to $P$ and $P'$ respectively via $\iota$.

We claim that the flex $P_0$ can be chosen so that $P-P_0$ and $P'-P_0$ have the same order.  There is nothing to show unless $3 \ndiv d$, $P-P_0$ has order $3d$, and $P'-P_0$ has order $d$, up to interchange of $P$ and $P'$.  Replacing $P_0$ by a different flex is tantamount to adding a nontrivial 3-torsion point $t$ to both $P-P_0$ and $P'-P_0$. We observe that for some choice of flex, both points have order $3d$. Otherwise for all nontrivial 3-torsion points $t$, either $d(P-P_0+t)=0$ or $d(P'-P_0+t)=0$, which is immediately seen to be impossible, recalling that $3\ndiv d$.

Having guaranteed that $P-P_0$ and $P'-P_0$ have the same order by a judicious choice of flex, then since $\mathrm{SL}(2,\mathbb{Z})$ acts transitively on the set of torsion points of $E_{\omega}\simeq (\RR/\ZZ)^2$ of any fixed order, we can find $A\in \mathrm{SL}(2,\mathbb{Z})$ so that $(m',n')=(m,n)A$.  We  then choose a path $A(t),\ 0\le t\le 1$ in $\mathrm{SL}(2,\mathbb{R})$ satisfying $A(0)=I$ and $A(1)=A$.  
Then the required path in $U_d$ can be taken to be
\begin{equation}
  \label{eq:thearc}
  t\mapsto \left(\iota\left(\left(m,n\right)A(t)\,^{\mathrm{t}}\omega\right),\iota
\left(E_{\omega\,^{\mathrm{t}}A\left(t\right)}\right)\right).
\end{equation}
\end{proof}
As discussed above, Lemma~\ref{lem:p2irred} establishes our Conjecture for $S=\mathbb{P}^2$.

\begin{Proposition}\label{prop:monodromic}
  Conjecture~\ref{conj:well-def} is true for $S=\mathbb{P}^2$.
\end{Proposition}

As the results of our calculations in this section are independent of $P$, we sometimes omit $P$ and write $m_\beta=m^P_\beta$.

\begin{Proposition}\label{prop:rel sev}
Let $\beta\in\hhh_2(S,\ZZ)$ be a curve class that contains an integral curve and choose $P\in E(\beta)$. Then there is a short
exact sequence
\[
\xymatrix{
0 \ar[r] & \hhh^0(\mc O_{S}(\beta-E)) \ar[r] & \hhh^0(\mc O_{S}(\beta)) \ar[r]^-\rest & \hhh^0(\mc O_{E}(wP)) \ar[r] & 0,
}
\]
where $\rest$ is the restriction map. Moreover, $h^{0}(\shO_S (\beta-E)) = p_a(\beta)$, 
$h^0(\shO_S(\beta))=\chi(\shO_S(\beta))=p_a(\beta)+w$ and $h^1(\shO_S(\beta))= h^2(\shO_S(\beta))=0$.
\end{Proposition}
\begin{proof}
Consider the short exact sequence
$$
0\to\mc O_{S}(-E)\to\mc O_{S}\to\mc O_{E}\to0.
$$
Since $\beta|_{E}\sim wP$, tensoring with $\mc O_{S}(\beta)$
yields
\begin{equation}\label{eq:resttoe}
0\to\mc O_{S}(\beta-E)\to\mc O_{S}(\beta)\to\mc O_{E}(wP)\to 0.
\end{equation}
Consider the maps induced on sections
\begin{equation}\label{Jinwon}
0 \to \HHH^0(\mc O_{S}(\beta-E)) \to \HHH^0(\mc O_{S}(\beta)) \to \HHH^0(\mc O_{E}(wP)),
\end{equation}
where the last map is the restriction map. By Serre duality, $\HHH^1(\shO_S(\beta-E))=\HHH^1(\shO_S(-\beta))$.
Denote by $\shO_\beta$ the structure sheaf $\shO_C$ of an integral curve of class $\beta$. From the long exact sequence induced from
\begin{equation*}
0 \to \mc O_S(-\beta) \to \mc O_S \to \mc O_\beta \to 0
\end{equation*}
we obtain $\HHH^0(\shO_S) = \HHH^0(\shO_\beta)$ and thus $\HHH^1(\shO_S(-\beta))\hookrightarrow\HHH^1(\shO_S)=0$. It follows that the restriction map in (\ref{Jinwon}) is surjective. 
Moreover, $\HHH^2(\shO_S(\beta-E))=\HHH^0(\shO_S(-\beta))=0$ by Serre duality. 
Thus it follows from Riemann-Roch and the adjunction formula that
\[
h^0(\shO_S(\beta-E))=\chi(\mc O_S(\beta-E))
=\frac{1}{2}(\beta-E) . \beta +1 = p_a(\beta). 
\]
Finally, from the long exact cohomology sequence associated to \eqref{eq:resttoe}, we find that $h^1(\shO_S(\beta))= h^2(\shO_S(\beta))=0$
and therefore $h^0(\shO_S(\beta))=\chi(\shO_S(\beta))$.
\end{proof}

\begin{definition}\label{def:linsys}
Set
\[
|\shO_S(\beta,P)| := \left\{ C \in |\beta| \, : \, C\cap E \supseteq wP \text{ as subschemes of } E  \right\}.
\]
\end{definition}

Up to scalar multiple, there is only one section in $\hhh^0(\mc O_{E}(wP))$ that vanishes to order $w$ at $P$. Denote the corresponding one-dimensional subspace of $\hhh^0(\mc O_{E}(wP))$ by $L_P$. By Proposition \ref{prop:rel sev},
\[
|\shO_S(\beta,P)|=\PP\left(\rest^{-1}(L_P)\right)
\]
is of dimension $p_a(\beta)$.

\begin{remark}
If $P$ is $\beta$-primitive, we identify $M_{\beta,P}$ with the set of rational curves  in $|\shO_S(\beta,P)|$ with only one branch at $P$. The set $M_{\beta,P}$ is finite by Proposition \ref{prop:finite}. 
\end{remark}

\begin{Proposition}[See also Theorem 3.10 of \cite{CGKT}]\label{propprop}
Let $\beta\in \hhh_2(S,\ZZ)$ be a curve class and $P$ a $\beta$-primitive point. Consider the blow-down $\pi : S\to S'$ of $S$ along some line classes $l$ such that $\beta.l=0$. Then $\pi(P)$ is $\pi_*\beta$-primitive on $\pi(E)$, and each curve $C\in M_{\beta,P}$ is isomorphic to exactly one curve $C'\in M_{\pi_*\beta, \pi(P)}$ and vice-versa. Furthermore, for each such $C$, a neighborhood of $C$ in $S$ is isomorphic to a neighborhood of $C'$ in $S'$ 
and $m^P_\beta=m^{\pi(P)}_{\pi_*\beta}$.

\end{Proposition}

\begin{proof}
It is straightforward to see that $\pi(P)$ is $\pi_*\beta$-primitive. 

Let $C\in|\shO_S(\beta,P)|$ and assume that $C$ does not contain $E$. Then $C$ is integral and does not meet any representative of $l$. Therefore, blowing down a line $l$ does not change members of $|\shO_S(\beta,P)|$ except for those containing $E$. In particular, it does not change the geometry in a neighborhood of elements of $M_{\beta,P}$ and $M_{\pi_*\beta,\pi(P)}$. 
Hence there is an isomorphism of moduli spaces $\msep\cong \ol{\mmm}^{\,\pi_*P}_{\pi_*\beta}(S',\pi(E))$ and $m^P_\beta=\nsep=\shN^{\pi_*P}_{\pi_*\beta}(S',\pi(E))=m^{\pi(P)}_{\pi_*\beta}$.

\end{proof}

\begin{Corollary}\label{cor}
Let $\beta\in \hhh_2(S,\ZZ)$ be a curve class containing an integral member and such that $p_a(\beta)\ge 1$. Each distinct two lines $l_1$ and $l_2$ with $\beta.l_1 = \beta.l_2 = 0$ are mutually disjoint. Let $\eta$ be the number of disjoint lines $l$ with $\beta.l=0$ and let $\pi : S\to S'$ be the del Pezzo surface obtained by blowing down these lines. Assume that $\beta\neq -K_{S_7}, \, -K_{S_8} \text{ or } -2K_{S_8}$ (which are ample). Then
$\pi_*\beta$ is very ample and $m^P_\beta=m^{\pi(P)}_{\pi_*\beta}$. 
\end{Corollary}
\begin{proof}

If two distinct lines $l_1$ and $l_2$  were to satisfy $\beta.l_1 = \beta.l_2 = 0$, then since $\beta^2>0$ by adjunction, it follows from the Hodge index theorem that $(l_1 + l_2)^2 < 0$ and therefore $l_1$ and $l_2$  are mutually disjoint.
Furthermore, $\pi_*\beta$ is very ample by Lemma \ref{lem:Rocco} and $m^P_\beta=m^{\pi(P)}_{\pi_*\beta}$ by Proposition \ref{propprop}. 
\end{proof}

\begin{remark}
By Proposition \ref{propprop}, the log BPS number $m_\beta$ should depend on $e(S)-\eta$ and not simply on $e(S)$. This is exactly what we observe in the prediction of Theorem \ref{thm:localcalc} (analogously \cite[Theorem 1.1]{CGKT} for the local BPS invariants). 
\end{remark}

\begin{Proposition}[See Proposition \ref{prop:logk3contr}]\label{propocontro}
Let $(S,E)$ be a regular log Calabi-Yau surface with smooth divisor and let $\beta\in\hhh_2(S,\ZZ)$ be a curve class. Let $C$ be an (integral) rational curve of class $\beta$ maximally tangent to $E$ at $P$, and denote the normalization map by $n:\pone\to C$. 
Then:
\begin{enumerate}
\item The map $n$ gives an isolated point of $\msep$, 
and contributes a positive integer to $\nsep$. 
\item If $C$ is immersed outside $P$, $[n]$ contributes $1$ to $\nsep$.
\item[(3)] If $n$ is smooth at $P$, then it contributes $\multc$ to $\nsep$.
\item[(4)] If $C$ is smooth at $P$ and has arithmetic genus $1$, then its contribution $\multc$ equals $e(C)$ and can be recovered from the associated elliptic fibration.
\end{enumerate}
Assume that $S$ is del Pezzo.
\begin{enumerate}
\item[(5)] Let $P\in E(\beta)$ be $\beta$-primitive. Assume that all $C\in M_{\beta,P}$ are smooth at $P$. Then
\[
m^P_\beta  = \sum_{C\in M_{\beta,P}} \multc.
\]
\end{enumerate}
\end{Proposition}

\begin{proof}
(1) 
By Proposition \ref{prop_description_log_maps_delpezzo}, 
$n$ lifts to a unique stable log map. 
By \cite[Example 7.1]{GS13} or more generally Proposition \ref{prop:Wise}, in a neighborhood of $[n]$, 
the moduli of stable log maps 
is isomorphic to a certain locally closed substack of the moduli 
of $1$-marked ordinary stable maps 
parametrizing maps $f$ with $f^*E=w.(\hbox{marked point})$. 
It follows from Proposition \ref{prop:finite} that $[n]$ is isolated. 
Since $n$ has no nontrivial automorphisms, 
this isolated component is a scheme, 
and the contribution of $[n]$ is given by its length. Here we use the fact that the virtual fundamental class of a non-reduced point is the fundamental class and its degree is the length of the point.

(2)
Let $x\in\pone$ be such that $n^{-1}(E)=\{x\}$. 
The deformation of $[n]$ in $\mse$ is governed 
by the complex of vector bundles $T^\bullet:=[T_{\pone}(-\log x)\overset{dn}{\to} n^*T_S(-\log E)]$ 
in degrees $-1$ and $0$. 
In particular, the tangent space is given by $R^0\Gamma(\pone, T^\bullet)$. 

If $n$ is immersive outside $x$, 
$dn$ is injective away from $x$.
If locally $x$ is given by $y=0$, then at $x$, $T_{\pone}(-\log x)$ is generated by $y\frac{\partial}{\partial y}$.
Moreover, $n^*T_S(-\log E)$ is generated by $s\frac{\partial}{\partial s}$ and $\frac{\partial}{\partial t}$ if in local coordinates $s,t$ at $f(x)$, $E$ is given by $s=0$. Then, $dn : y\frac{\partial}{\partial y} \mapsto w \cdot s\frac{\partial}{\partial s}$ is injective as well and thus $dn$ is injective everywhere.

By calculating the degree, 
we see that $T^\bullet$ is quasi-isomorphic to $\mathcal{O}_{\pone}(-1)$. 
So $[n]$ is infinitesimally rigid.

(3)
Let $M_{[n]}$ be the connected component of $\mse$ at $[n]$ and 
denote by $m(C)$ the length of $M_{[n]}$. The multiplicity of $[n]$ as a stable map to $C$, i.e.\ the length of the scheme $\mathrm{M}_{0,0}(C, [C])$, was calculated
in \cite{FGS}, see also
\cite{Shende}, to be $\multc$.

As we noted above, the moduli of log structure is trivial and $M_{[n]}$ can be considered as a subscheme of $\mathrm{M}_{0,0}(S, [C])$ by \cite[Example 7.1]{GS13} or Proposition \ref{prop:Wise}. In order to show that $m(C)$ is the stable map multiplicity of $[n]$, we need to show that deformations of $[n]$ in $M_{[n]}$ factor (scheme-theoretically) through $C$.
Note that since $C$ is assumed to be smooth at $P$, in the notation of \cite[Example 7.1]{GS13}, this guarantees that $u_{x}=\mu_x$ at $P$ remains true in a deformation.\footnote{On the converse, say $C$ has a cusp at $P$. Then $[n]$ has a first order deformation as a stable map, resolving the cusp. Under this deformation the pullback of $E$ is no longer a multiple of $\mu_x$ (e.g.\ if $\mu_x=2$, resp.\ $3$, and $x$ is defined by $t=0$, then
the pullback of $E$ is defined by $t^2+\epsilon$,
resp. $t^3+\epsilon t$, for an infinitesimal parameter $\epsilon$).} 

Consider the open sub-linear system
\[
|\shO_S(\beta,P)|^\circ := \left\{ C \in |\beta| \, : \, C\cap E = wP  \right\} \subseteq |\shO_S(\beta,P)|
\]
and denote by $\rcjp$ the relative compactified Jacobian of the universal curve over $|\shO_S(\beta,P)|^\circ$.
Then \cite[Theorem 1.11]{CGKT3} states that $\rcjp$ is smooth at each point of the fiber $\ol{\pic}^0(C)$. Using this result, the statement follows from \cite[Section F, proof of Theorem 2]{FGS}.

(4) In this case, we can recover (3) directly. Outside the member of $\beta$ with a singularity at $P$ (see \S\ref{sec:logg1}), the curve family can be identified with the associated elliptic fibration, and with $\rcjp$ at integral fibers. The total space of the elliptic fibration is given by successive blowups at $P$, so it is a nonsingular surface.

(5) follows from (3).
\end{proof}

It is shown in \cite[\S 6]{Shende} that in an appropriate sense and up to sign the contribution of an integral rational curve $C$ to $n_\beta$ is given by $e(\ol{\pic}^0(C))$. In light of Conjecture \ref{conj3}, Proposition \ref{propocontro} states that in a normalized way, $C$ contributes the same to the local and log BPS numbers.

\section{Calculations at primitive points}\label{sec:calc}

In this section, we compute $m^P_\beta$ for all del Pezzo surfaces $S$ with smooth anticanonical divisor $E$, classes $\beta$ of arithmetic genus up to 2 and $\beta$-primitive points $P$. In genus 2, we will assume that $(S,E)$ is general.

\subsection{Strategy of proof and summary of results}
Let $P\in E(\beta)$ be $\beta$-primitive. Recall that $w=\beta.E$ and that $M_{\beta,P}$ denotes the finite set of integral rational curves of class $\beta$ maximally tangent to $E$ at $P$ that have only 1 analytic branch on $E$.

By Definition \ref{def:prim} of $\beta$-primitivity and Proposition \ref{prop:finite}, $\msep$ is a disjoint union of (possibly non-reduced) points with each point corresponding to the unique log map induced from the normalization map $n : \pone \to C$ for $C\in M_{\beta,P}$ (Proposition \ref{prop_description_log_maps_delpezzo}(2)). Therefore the virtual fundamental class $[\msep]^\mathrm{vir}$ is the fundamental class of $\msep$ and $m_\beta^P=\nsep$ is the sum of the lengths of the points. In Proposition \ref{propocontro} we calculated the contributions to $m_\beta^P$ in the two situations where $C$ is immersed outside of $P$ and when $n$ is smooth at $P$. 

Our strategy of proof therefore consists of finding all elements of $M_{\beta,P}$ and calculating their contribution to $m_\beta^P$. In the case of $\ptwo$, we note two alternative approaches to this calculation: by tropical correspondence in the associated scattering diagram in \cite{Gra} and by degenerating $E$ to the union of the coordinate axes in \cite{BN}.

We make some comments about the possible curve classes.

\begin{Lemma}
Let $\beta$ be a curve class on $S$. Then $\beta$ falls into one of the following categories.
\begin{enumerate}
\item $\beta$ does not contain an integral curve. This occurs for instance when $p_a(\beta)<0$.
\item $\beta$ contains an integral curve and $p_a(\beta)=0$. Among these, line classes are not nef, conic classes are nef but not big and the rest are nef and big.
\item $\beta$ contains an integral curve and $p_a(\beta)\geq1$. In this case containing an integral curve is equivalent to being nef and big.
\end{enumerate}
\end{Lemma}
\begin{proof}
By \cite[(7.6) Corollary]{Mum}, if $p_a(\beta)<0$, then $\beta$ is not irreducible. By \cite[\S 2.3 (P7)]{Knu}, nef and big divisors admit a smooth irreducible member, so nef and big always implies containing an integral curve.
If $S=\pone\times\pone$, then the statements can be checked numerically.

Assume $S$ is a blowup of $\ptwo$. 
By \cite[Proposition 3.2]{Rocco}, effective divisors fall into 2 categories: A) nef and big ones; B) line classes, conic classes of the form $h-e_i$ and classes $\beta$ that admit a decomposition $\beta = \beta' + l$ for $\beta'$ effective and $l$ a line class. The last category includes the remaining conic classes. Line and conic classes have arithmetic genus 0.

By \cite[Theorem 4.7]{Rocco} nef divisors are globally generated and hence admit a smooth member.
By the remark after Corollary 3.3 of \cite{TVV}, nef but non-big divisors are multiples of conic classes $kD$. For $k>1$, their arithmetic genus is negative and they are not integral (a general member will be a disjoint union of $k$ curves of class $D$).
Then, non-nef divisors necessarily have a line class as a component, i.e.\ can be written as $\beta = \beta' + l $ for $\beta'$ effective, $l$ a line class and $\beta'.l\leq0$. If $\beta'\neq0$, then any curve of class $\beta$ contains $l$, hence is not irreducible.
\end{proof}

\begin{thm} \label{thm:logcalc}
Let $S$ be a del Pezzo surface and let $E$ be a smooth anticanonical divisor. Let $\beta$ be a curve class. Let $\eta$ be the number of line classes $l$ such that $\beta. l =0$. Recall that $S_8$ denotes the del Pezzo surface obtained by blowing up $\ptwo$ in 8 general points.
\begin{enumerate}
\item[(0)] If $\beta$ does not contain an integral curve, such as when $p_a(\beta)<0$, then $m_\beta^P=0$.
\end{enumerate}
Assume that $\beta$ contains an integral curve.
\begin{enumerate}
\item If $p_a(\beta)=0$, then $m^P_\beta= 1$.
\item If $p_a(\beta)=1$ and $\beta\ne -K_{S_8}$, then $m^P_\beta= e(S)-\eta$.
\item If $\beta=-K_{S_8}$, then $m^P_\beta=12$. 
\item If $p_a(\beta)=2$ and $\beta\neq-2K_{S_8}$, then $m^P_\beta=  \binom{e(S)-\eta}{2}+5$.
\item If $\beta = -2K_{S_8}$, then $m^P_\beta=66$. 
\end{enumerate}
\end{thm}

We prove Theorem \ref{thm:logcalc} in the remainder of this section.
(0) is straightforward:
When there is no integral curve of class $\beta$, then $M_{\beta,P}$ is empty (as $P$ is $\beta$-primitive) and $m^P_\beta=0$. Henceforth we assume that $\beta$ contains an integral curve. In particular, in this case Proposition \ref{prop:rel sev} holds.

\subsection{Arithmetic genus 0}

Assume that $p_a(\beta)=0$ and let $P$ be a $\beta$-primitive point. By Proposition \ref{prop:rel sev}, $\HHH^{0}(\mc O_{S}(\beta-E))=0$ and $\HHH^{0}(\mc O_{S}(\beta))\simeq\HHH^{0}(\mc O_{E}(wP))$.
There is thus a (unique up to scaling) non-zero section $s\in\HHH^{0}(\mc O_{E}(wP))$ that vanishes at $P$ of order $w$.
By $\beta$-primitivity,  the corresponding curve is necessarily integral, thus it is isomorphic to $\pone$. Hence $m^P_{\beta}=1$.

\subsection{Arithmetic genus 1} \label{sec:logg1}

Assume now that $p_a(\beta)=1$ and let $P\in E(\beta)$ be a $\beta$-primitive point. Assume first that $w>1$.
Let $\eta$ be the number of disjoint lines $l$ with $\beta.l=0$. We blow down all such lines in $S$, yielding a del Pezzo surface $\pi : S \to S'$ with divisor $E'$ and a curve class $\beta':=\pi_*\beta$.  Then $e(S')=e(S)-\eta$ and by Corollary \ref{cor}, $\beta$ is ample and $m^P_\beta=m^{P'}_{E'}$ for $P':=\pi(P)$. By \cite[Lemma 4.3]{CGKT} or Proposition \ref{prop:rat} (3), $\beta'=[E']$. 

It follows from Proposition \ref{prop:rel sev} that the linear system $\Lambda:=|\shO_{S'}(E',P')|$ of Definition \ref{def:linsys} is of dimension 1, i.e.\ is a linear pencil on $S'$. 
Note that any member of $\Lambda\setminus\{E'\}$ has at worst a node or a cusp, since it is integral by $\beta$-primitivity and is of arithmetic genus $1$; 
$E'$ itself is of course nonsingular. We will first describe the pencil and then consider its associated universal elliptic fibration.

Denote by $\sigma_E$ the section of $\shO(E')$ giving $E'$ and by $\sigma$ another section. Then the elements of $\Lambda$ correspond to the zero loci $C_{[s:t]}$ of $s\sigma + t\sigma_E$ for $[s:t]\in\pone$. Restricting to an open affine neighborhood $U$ of $P'$, there are functions $f:=\sigma|_U$ and $g:=\sigma_E|_U$ such that $C_{[s:t]}\big{|}_U$ is given by $sf+tg=0$. We write $f_{P'}$ and $g_{P'}$ for their stalks at $P'$.

Let $\mathrm{T}^\vee_{S',P'}:=m_{P'}/m_{P'}^2$  be the cotangent space of $S'$ at $P'$. The cotangent space $\mathrm{T}^\vee_{E',P'}$ of $E'$ at $P'$ can be identified with the 1-dimensional subspace of $m_{P'}/m_{P'}^2$ cut out by $g_{P'} \mod m_{P'}^2$ with natural linear quotient map $\mathrm{res} : \mathrm{T}^\vee_{S',P'} \to \mathrm{T}^\vee_{E',P'}$.

Consider now $\varphi(s,t) := \diff (sf_{P'}+tg_{P'})\in \mathrm{T}^\vee_{S',P'}$. As $w>1$, $C_{[s:t]}$ is tangent to $E'$ at $P'$ and $sf_{P'}+tg_{P'}$ is a contant multiple of $g_{P'} \mod m_{P'}^2$. 
Hence $\varphi(s,t)$ lies in the 1-dimensional kernel of $\mathrm{res}$. Then $\varphi(0,1)\neq0$. Therefore $\varphi(s,t)=s\,\varphi(1,0)+t\,\varphi(0,1)$, for $[s:t]\in\pone$, spans all of $\ker(\mathrm{res})$ and there is exactly one choice of $[a:b]$ such that $\varphi[a,b]=0$. 
This value corresponds to the unique member $D:=C_{[a:b]}$ of $\Lambda$ that is nodal or cuspidal at $P'$.
If $D$ is nodal at $P'$, one of its branches meets $E'$ with multiplicity $1$, and the other branch meets $E'$ with multiplicity $w-1$.

Next, to obtain the associated universal elliptic fibration, we blow up $w$ times along the inverse images of $P'$ in the strict transforms of $E'$. Assume first that $D$ is nodal at $P'$. The first blowup will separate the two branches of $D$ and the total transform of $D$ becomes a cycle of two $\pone$'s. Each successive blowup but the last will introduce another $\pone$ in the cycle. After the first $w-1$ blowups, the preimage $\widetilde{D}$ of $D$ consists of a cycle of $w$ $\pone$'s with topological Euler characteristic $w$. The last blowup then separates all of the curves in our pencil and we obtain a family $\mc C\to\pone$, 
where the last exceptional divisor maps isomorphically to the base $\pone$. 

We apply the same procedure for $D$ cuspidal at $P'$. In this case $\widetilde{D}$ is a chain of $\pone$'s and the Euler characteristic of $\widetilde{D}$ is $w+1$.

In both cases, $\widetilde{D}$ is a fiber of $\mc C\to\pone$. The other fibers are members of $\Lambda$ other than $D$. Using Proposition \ref{propocontro}(4), we can calculate $m^P_{\beta}$ by computing the topological Euler characteristic of $\mc C$. 
For a smooth curve $C$ of the pencil, $e(C)=0$. If $C$ is nodal, then $\multc=e(C)=1$ and if $C$ is cuspidal, $\multc=e(C)=2$. 
In case $D$ is nodal at $P'$,
\[
e(\mc C)=\#\left\{ \text{nodal fibers}\right\} +2\cdot\#\left\{ \text{cuspidal fibers}\right\} +w,
\]
where we do not count $D$ among $\left\{ \text{nodal fibers}\right\}$.
In case $D$ is cuspidal at $P'$,
\[
e(\mc C)=\#\left\{ \text{nodal fibers}\right\} +2\cdot\#\left\{ \text{cuspidal fibers}\right\} +w+1.
\]
On the other hand, 
\[
e(\mc C)=e(S')+\#\left\{ \text{blowups}\right\} = e(S')+w.
\]
If $D$ is nodal at $P'$, then $D\not\in M_{\beta', P'}$ and does not contribute to $m^P_\beta$. By Proposition \ref{propocontro}(4),
\[
m^P_\beta=\#\left\{ \text{nodal fibers}\right\} +2\cdot\#\left\{ \text{cuspidal fibers}\right\} = e(S') =e(S)-\eta.
\]
If $D$ is cuspidal at $P'$, then
$D$ is also in $M_{\beta', P'}$ and contributes 1 to $m^P_\beta$ by Proposition \ref{propocontro}(2).
Therefore
\[
m^P_\beta=1+\#\left\{ \text{nodal fibers}\right\} +2\cdot\#\left\{ \text{cuspidal fibers}\right\}=e(S')=e(S)-\eta
\]
as well.

Finally, we consider the case that $w=1$. Then it follows that $\beta=-K_{S_{8}}$. Since the intersection multiplicity is 1, there is no curve
in the pencil $\Lambda$ that is singular at $P$. To obtain the universal family, we only need to blow up once and obtain that
\[
e(S_8)+1=e(\mc C)=\#\left\{ \text{nodal fibers}\right\} +2\cdot\#\left\{ \text{cuspidal fibers}\right\} ,
\]
so that $m^P_{\beta}=e(S_8)+1=12$.  Note that $\eta=0$ in this case, since $\beta.\ell=1$ for any line.

\begin{remark}
In light of the previous argument, note that we expect that for general $(S,E)$, curves with a cusp at the marked point are avoided.
\end{remark}

\begin{remark}\label{rem:52}
In the case of $\ptwo$, we can rule out cuspidal degree 3 curves meeting $E$ in a $3h$-primitive point, as was noted in the proof of \cite[Proposition 1.5]{tak compl}. Indeed, suppose that there is a degree 3 cuspidal curve $C$ meeting $E$ in a 9-torsion point $P$, where we take the group law on $E$ obtained by choosing a flex point as $0$. Then $C$ also has a group law, and is isomorphic to $\mathbb{G}_a$ away from its cusp. The zero of $C$ is a flex point. Moreover, $E$ induces the equation $9P=0$ on $C$. Since $C$ has no torsion, $P=0$ is a flex point, which is impossible since $C$ and $E$ have 9-tuple intersection at $P$ and since $P$ is not a flex point for $E$.
\end{remark}

\subsection{Approach for arithmetic genus 2 and higher}\label{sec:g2high}
Assume that $(S,E)$ is general and let $\beta$ be a curve class with $p_a(\beta)=2$ containing an integral member. Let $P\in E(\beta)$ be a $\beta$-primitive point and let $\eta$ be the number of disjoint lines $l$ with $\beta.l=0$. By Corollary \ref{cor}, after blowing down these lines, we may assume that $\beta$ is very ample, except when $\beta=-2K_{S_8}$, which we will treat at the end. We will also treat $\pone\times\pone$ separately. Finally, note that there are no genus 2 classes for $\ptwo$.

Our approach consists in counting curves in $M_{\beta,P}$ and showing that they are all nodal away from $P$ by degeneration technique. This was done in \cite{tak compl}.
It would take us too far afield to explain \cite{tak compl} in detail, so we content ourselves with giving a brief introduction to its methods.

An $\AA^1$-curve on $(S, E)$ can be identified with a morphism $f: \pone\to S$ 
such that $f^*E=w[\infty]$. 
We start with a certain degeneration of $(\pone, w[\infty])$ 
(\cite[\S3, second paragraph]{tak compl}). 
This is given by a flat family $p^{(w)}: \shC^{(w)}\to U^{(w)}\cong \pone$ of curves 
and a Cartier divisor $\shZ^{(w)}$ on $\shC^{(w)}$ 
such that the following holds: 
\begin{itemize}
\item
For $s\in U^{(w)}\setminus\{\infty\}$, 
the fiber $(\shC^{(w)}_s, \shZ^{(w)}_s)$ is isomorphic to $(\pone, w[\infty])$. 
\item
The fiber $\shC^{(w)}_\infty$ has two components $C_1$ and $C_2$ 
each isomorphic to $\pone$, 
intersecting at a point $Q$, 
and $\shZ^{(w)}_\infty|_{C_1}=(w-1)Q$ and $\shZ^{(w)}_\infty|_{C_2}=Q$. 
\end{itemize}

On the other hand, 
we also consider a family of target spaces. 
Let $\shS\to T$ be a family of smooth projective rational surfaces, 
$\shE$ a relative smooth anticanonical divisor, 
$\shB$ a relative divisor class which will correspond to $\beta$ 
and $\shP$ a closed subscheme of $\shE$ which gives a section of $\shS\to T$ 
such that $\shB|_\shE\sim w\shP$ as relative divisor classes on $\shE$. 
Let $\shS_t$ denote the fiber of $\shS$ at $t\in T$ etc. 

\begin{construction}[Section 3 in \cite{tak compl}, in particular Definition 3.3 
and the paragraph before Lemma 3.10]
Given $(\shS, \shE, \shP, \shB)$ as above, 
there is a moduli space $M(\shS,\shE,\shP,\shB)$ 
which represents the functor 
\[
(\hbox{a scheme $U$ over $T$})\mapsto 
(\hbox{the set of isomorphisms classes $(\shC, \shZ, f)$}), 
\]
where 
\begin{itemize}
\item
$\shC$ is a flat family of curves over $U$, $\shZ$ is a Cartier divisor on $\shC$ and 
$(\shC, \shZ)$ is isomorphic to a family induced from $(\shC^{(w)}, \shZ^{(w)})$ 
by a morphism $U\to U^{(w)}$, 
\item
$f: \shC\to \shS$ is a morphism over $T$ which satisfies the following: 
\begin{itemize}
\item
For any $u\in U$ over $t\in T$, the restriction $f_u: \shC_u\to \shS_t$ 
is generically injective and $(f_u)_*[\shC_u]=\shB_t$. 
\item
$f^*\shE=\shZ$. 
\end{itemize}
\end{itemize}
An isomorphism between $(\shC, \shZ, f)$ and $(\shC', \shZ', f')$
is an isomorphism $\shC\to \shC'$ which maps $\shZ$ to $\shZ'$ 
and commutes with $f$ and $f'$. 

The space $M(\shS,\shE,\shP,\shB)$ contains 
an open subspace $M_0(\shS,\shE,\shP,\shB)$ 
parametrizing generically injective morphisms $\pone\to \shS_t$ 
that meet $\shE_t$ with maximal tangency at $\shP_t$.
The complement $M(\shS,\shE,\shP,\shB)\setminus M_0(\shS,\shE,\shP,\shB)$ 
parametrizes 
morphisms from $\pone\cup_{\mathrm{pt}}\pone$ to $\shS_t$ 
where the first component meets $\shE_t$ at $\shP_t$ with multiplicity $w-1$ 
and the second with multiplicity $1$.

\end{construction}

\begin{Lemma}[Lemma 3.4 of \cite{tak compl}]\label{lem:2cptmod}
Consider the case of $(S,E,P,\beta)$, i.e. a family over a point.
\begin{enumerate}
\item Geometric points of $M_0(S,E,P,\beta)$ bijectively correspond to elements of $M_{\beta,P}$.
\item A point of $M_0(S,E,P,\beta)$ corresponding to $C\in M_{\beta,P}$ is reduced if the normalization map $\pone \to C$ is immersive away from $P$.
\end{enumerate}
\end{Lemma}

By Lemma \ref{lem:2cptmod} and Proposition \ref{propocontro}(2), and under the assumption that all curves of $M_{\beta,P}$ are immersed away from $P$, $m^P_\beta=\big{|}M_{\beta,P}\big{|}=\big{|}M_0(S,E,P,\beta)\big{|}$.

Now we construct certain surfaces and divisor classes. 
Let $d$, $(a_i)_{i=1}^r$ and $(b_j)_{j=1}^m$ be positive integers 
and write $w=w(d; (a_i); (b_j))=3d-\sum a_i-\sum b_j$, 
which we assume to be positive. 
We take a smooth plane cubic $E_0$ and points $P_1, \dots, P_r\in E_0$ and $P\in E_0$ 
such that $dh|_{E_0}\sim \sum a_iP_i+(w+\sum b_j)P$. 
We impose a kind of ``primitivity'' on $P$ 
(\cite[Definition 3.6]{tak compl}). 
Let $S$ be the surface given by blowing up $\ptwo$ at $P_1, \dots, P_r$ and 
then successively $m$ times at $P$ and the corresponding points on the proper transforms 
of $E_0$. 
The proper transform of $E_0$ on $S$ is denoted by $E$. 
Write $h$ for the hyperplane class, $e_i$ for the class of the exceptional curve over $P_i$, 
and $f_j$ for the class of the exceptional curve of the $j$-the blowup at $P$. 
Let $\beta=dh-\sum_{i=1}^r a_i e_i-\sum_{j=1}^m b_j f_j$. 
Then, in \cite{tak compl}, the integers $n(d;(a_i);(b_j))$ are defined as $\deg M_0(S,E,P,\beta)$ 
for a general choice of $E_0, P_1, \dots, P_r$. If $m=0$, i.e.\ there are no $(b_j)$, then we simply write $n(d;(a_i);)$ for $\deg M_0(S,E,P,\beta)$.

If $r\leq 8$ and $m=0$, 
then $S$ is a general del Pezzo surface $S_r$. 
In this case, under the assumption that the curves of $M_{\beta,P}$ are immersed away from $P$, 
\[
n(d;(a_i);)=\deg M_0(S,E,P,\beta)=\big{|}M_0(S_r,E,P,\beta)\big{|}=m^P_\beta, 
\]
the BPS numbers we are looking for. 

To derive the main recursive formula, 
we consider a certain degeneration of $(S,E,P,\beta)$ with $m=0$ (i.e. no $b_j$'s). 
Over a smooth curve germ $0\in \Delta$, 
we consider families of $P_1, \dots, P_r, P\in E_0$ such that 
$P$ meets $P_r$ over $0$. 
Then we have a family $(\shS,\shE,\shP,\shB)$ over $\Delta$. 

For a general choice of the relevant data, 
it turns out that $M(\shS,\shE,\shP,\shB)$ is finite and flat 
and that it coincides with $M_0(\shS,\shE,\shP,\shB)$ outside $0$. 
Over $0$, $(\shS_0, \shE_0, \shP_0, \beta)$ can be identified with the case of $m=1$ and $(d; (a_i)_{i=1}^{r-1} ; a_r)$, and each element of $M(\shS_0, \shE_0, \shP_0, \beta)$ is 
either 
\begin{enumerate}
\item
an $\AA^1$-curve on $(\shS_0, \shE_0)$, 
\item
the sum of $F$ and an $\AA^1$-curve on $(\shS_0, \shE_0)$ in the case $w\geq 2$. or 
\item
the sum of $F$ and a rational curve disjoint from $\shE_0$ in the case $w=1$,
\end{enumerate}
where $F$ is the exceptional curve of the blowup at $P$. 
This gives  \cite[Theorem 3.8 (3)]{tak compl}: 
Under certain technical conditions which are satisfied in our case, 
\[
n(d; (a_i)_{i=1}^r; ) = n(d; (a_i)_{i=1}^{r-1}; a_r) + \mu\cdot n(d; (a_i)_{i=1}^{r-1}; a_r + 1), 
\]
where $\mu$ is $a_r+1$ if $w(d; (a_i)_{i=1}^r; )=1$ and $1$ otherwise. Note that when $w(d; (a_i)_{i=1}^r; )=1$, the rational curve and $F$ in (3) intersect at $a_r+1$ points so there are $a_r+1$ choices of 
morphisms from $\pone\cup_{\mathrm{pt}}\pone$ to this sum. 
If $a_r=1$, 
the first term of the right hand side can be related to the invariants with $m=0$ 
by \cite[Theorem 3.8]{tak compl}: 
\[
n(d; (a_i)_{i=1}^{r-1}; 1) = n(d; (a_i)_{i=1}^{r-1}; ). 
\]
On the other hand, larger values of $a_i$ and $b_j$ means 
smaller arithmetic genus, 
which makes the calculation easier. 

Finally, 
if $w(d; (a_i)_{i=1}^r; ) = 1$, 
then the tangency condition is automatically satisfied 
for curves in the class $dh-\sum a_ie_i$, 
and $n(d; (a_i); )$ is nothing but the number of rational curves. 

\begin{remark}
In the modern language, 
it should be possible to 
replace the use of $M(\shS,\shE,\shP,\shD)$ 
by the deformation invariance of log Gromov-Witten invariants 
and Theorem \ref{thm:takmult}, 
although the analysis of degenerate curves is still necessary. 
\end{remark}

Using these facts, 
Section 4 of \cite{tak compl} provides tables of $n(d;(a_i);(b_j))$. It is also checked that the curves relevant to us are immersed away from $P$. These tables contain all the curve classes of arithmetic genus 2 of $S_r$ for $r\leq 8$. Examining the numbers, for $\beta\neq -2K_{S_8}$, we find that 
\[
m_\beta^P = \binom{e(S_r)-\eta}{2} + 5.
\]
For $\beta=-2K_{S_8}$, $m^P_\beta=n(6,2^8;)$ is given in \cite[p. 21]{tak compl} to be $x-24$, where $x=90=n(6; 2^8, 1; )$ is found by calculating $n(6; ;)$ in two different ways \cite[proof of Cor. 4.5]{tak compl}. Therefore $m^P_\beta=66$.

It remains to consider $\pone\times\pone$. Up to permuting the generators $h_1$ and $h_2$, the only genus 2 class is $2h_1+3h_2$. Blowing up $\pone\times\pone$ in a general point, $2h_1+3h_2$ pulls back to a (non-very ample) class $\beta'$ on $S_2$ and the exceptional divisor for the blowup is the unique line class not meeting $\beta'$. By Proposition \ref{propprop}, $m^P_\beta=m^{P'}_{\beta'}$ for $P'$ the preimage of $P$. By the calculations of \cite{tak compl},
\[
m_{\beta'} = \binom{e(S_2)-1}{2} + 5 = \binom{e(\pone\times\pone)}{2} + 5
\]
as expected.

\begin{remark}
Note that in all calculations that we encountered, $m^P_\beta$ was given in terms of polynomials of topological numbers associated to $S$. 
\end{remark}

This finishes the proof of Theorem \ref{thm:logcalc}. Combining the calculations of this section and the ones of \cite{tak compl} with Theorem \ref{thm:localcalc} (Theorem 1.1 in \cite{CGKT}) proves Theorem \ref{thmthm}.

\section{Multiple covers and loop quiver DT invariants} \label{sec:loop}

Let $(S,E)$ be a log Calabi-Yau surface pair.
Let $C$ be an integral nodal rational curve in $S$ that meets $E$ with maximal tangency at $P$.  
Denote by $[C]$ the class of $C$, as usual. For now, neither is $S$ assumed to be del Pezzo, nor is $P$ assumed to be $[C]$-primitive.
By Proposition \ref{propocontro}(1), the normalization map $[\pone\to C]$ is an isolated 0-dimensional component of $\ol{\mmm}_{[C]}(S,E)$.

In the previous sections, we typically considered the moduli space $\ol{\mmm}_{[C]}(S,E)$ of genus 0 basic stable log maps of maximal tangency and class $[C]$ and its associated log GW invariant $\mc{N}_{[C]}(S,E)$.
By Proposition \ref{prop:finite}, the components $\ol{\mmm}^P_{[C]}(S,E)$ corresponding to maps at $P$ are disjoint from the other components and we may consider the invariant $\mc{N}^P_{[C]}(S,E)$ at $P$. 
By Proposition \ref{propocontro}(2), the normalization map $[\pone\to C]$ has length 1 and contributes 1 to $\mc{N}^P_{[C]}(S,E)$.

Let $l\in\NN$ and consider the moduli space $\ol{\mmm}_{l[C]}(S,E)$. This moduli space has a distinguished component of dimension $l-1$ that corresponds to $l:1$ multiple covers of $C$ that are totally ramified at $P$, in other words to the maps $[f : D\to C]\in\ol{\mmm}_{l[C]}(S,E)$ with $f_*D=lC$ as cycles. Denote by $\contr(l,C)$ the contribution of this component to the invariant $\mc{N}^P_{l[C]}(S,E)$. In the language of relative stable maps (equivalent to stable log maps by \cite{AMW}), it is defined and calculated in \cite[Section 6]{GPS10} to be
\begin{equation}\label{eq:BPS}
\contr(l,C)=\dfrac{1}{l^2}\dbinom{l(C.E-1)-1}{l-1},
\end{equation}
with the identification $\binom{-1}{l-1}=(-1)^{l-1}$.

Furthermore, we defined the \emph{log BPS number at $P$} of class $l[C]$, $m^P_{l[C]}$. Write $[C]=g\beta$ for $\beta$ a primitive curve class. Then $m^P_{l[C]}=m^P_{lg\beta}$ is defined by the recursive formula
\begin{equation}\label{eq:rec}
\mc{N}^P_{h\beta}(S,E) =  \sum_{k|h}     
\frac{(-1)^{(k-1) \, \frac{h}{k} \, (\beta.E)}}{k^2} \, m^P_{\frac{h}{k}\beta}.
\end{equation}
where $m^P_{\beta'}=0$ if there is no curve of class $\beta'$ that is maximally tangent to $E$ at $P$.
A natural question is to ask what the contribution of multiple covers over $C$ to $m^P_{l[C]}$ is. To answer this, we invert \eqref{eq:rec},
\[
m^P_{h\beta} =  \sum_{k|h}     
\frac{(-1)^{(k-1) \, \frac{h}{k} \, (\beta.E)}}{k^2} \, \mu(k) \, \mc{N}^P_{\frac{h}{k}\beta}(S,E),
\]
and in particular
\begin{equation}\label{eq:recinv}
m^P_{l[C]}=m^P_{lg\beta} =  \sum_{k | lg}     
\frac{(-1)^{(k-1) \, \frac{lg}{k} \, (\beta.E)}}{k^2} \, \mu(k) \, \mc{N}^P_{\frac{lg}{k}\beta}(S,E),
\end{equation}
In the right hand side of \eqref{eq:recinv}, we pick out the $\mc{N}^P_{\frac{lg}{k}\beta}(S,E)$ where the moduli space has a component corresponding to multiple covers over $C$. This happens exactly when $k|l$, in which case we have $\frac{l}{k}:1$ covers over $C$ that contribute $\contr(\frac{l}{k},C)$ to $\mc{N}^P_{\frac{l}{k}[C]}(S,E)$.

\begin{definition}\label{def:loopy}

Let $l\in\NN$. Define the multiple cover contribution $\contr^{\bps}(l,C)$ of $C$ to $m^P_{l[C]}$ to be
\[
\contr^{\bps}(l,C) :=  \sum_{k|l}   \frac{(-1)^{(k-1)l(C.E)/k}}{k^2} \, \mu(k) \, \contr(l/k,C).
\]
\end{definition}
When $C$ is nodal, by definition $\contr^{\bps}(1,C)=\contr(1,C)=1$. This is expected since there are not multiple covers.

We will express these contributions in terms of generalized Donaldson-Thomas (DT) invariants of loop quivers.
Motivated by the framework of Kontsevich-Soibelman in \cite{KS}, Reineke in \cite{rein} and \cite{Rec} calculated these invariants. We state Reineke's calculation in Theorem \ref{thm:Re} below and briefly survey the definition of the DT invariants in this particular setting.

Fix $m\geq 1$ and consider the $m$-loop quiver, consisting of one vertex and $m$ loops.  The associated framed $m$-loop quiver $L_m$ contains an additional vertex and an arrow directed towards the original vertex. The quiver $L_m$ is depicted below.

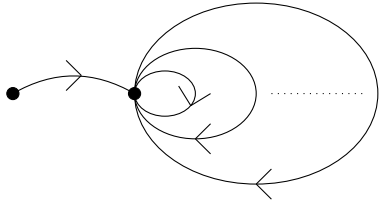
\vspace{0.2cm}
\begin{center}
\begin{tikzpicture}[scale=2]
\draw (-3.2,0) circle [x radius=0.2, y radius=0.15];
\draw (-3.03,-0.08) -- (-2.9,0);
\draw (-3.03,-0.08) -- (-3.11,0.05);
\draw [dotted] (-2.5,0) -- (-1.9,0);
\draw [fill] (-4.2,0) circle [radius=0.04];
\draw [fill] (-3.4,0) circle [radius=0.04];
\draw (-4.2,0) to [out=30,in=150] (-3.4,0);
\draw (-2.6,0) circle [x radius=0.8, y radius=0.6];
\draw (-2.6,-0.6) -- (-2.5,-0.5);
\draw (-2.6,-0.6) -- (-2.5,-0.7);
\draw (-3.75,0.12) -- (-3.85,0.22);
\draw (-3.75,0.12) -- (-3.85,0.02);
\draw (-3,0) circle [x radius=0.4, y radius=0.3];
\draw (-3,-0.3) -- (-2.9,-0.2);
\draw (-3,-0.3) -- (-2.9,-0.4);
\end{tikzpicture}
\end{center}
Let $n\geq 0$. A representation of $L_m$ of dimension vector $(1,n)$ is represented by
\[
(v,(\varphi_1,\dots,\varphi_m))\in\CC^n\oplus\mmm_n(\ic)^{\oplus m}.
\]
Denote by $\CC\langle\psi_1,\dots,\psi_m\rangle$ the free algebra on $m$ generators. Then $(v,(\varphi_1,\dots,\varphi_m))$ is \emph{stable} if $\CC\langle\varphi_1,\dots,\varphi_m\rangle v=\CC^n$, i.e.\ $v$ is cyclic for the representation of the free algebra on $\CC^n$.

Denote by
\[
U \subset \ic^n\oplus\mmm_n(\ic)^{\oplus m}
\]
the open subset of stable representations of $L_m$ of dimension vector $(1,n)$.  Then
$\gl_n(\ic)$ acts on $U$ via
\[
g\cdot (v,\varphi_i) = (gv,g\varphi_ig^{-1}).
\]
The geometric quotient for this action is the \emph{noncommutative Hilbert scheme} 
$\hilb_n^{(m)}$. Consider the generating function 
\[
F^{(m)}(t):= \sum_{n\geq 0} \chi\left(\hilb_n^{(m)}\right) t^n \in \iz[[t]].
\]
Since $F(0)=1$, $F(t)$ admits a product expansion.

\begin{definition}[Definition 3.1 in \cite{rein}, following \cite{KS}]
Define the \emph{generalized Donaldson-Thomas invariants  $\dt_n^{(m)}\in\iq$ of $L_m$} through the product expansion (Kontsevich-Soibelman wall-crossing formula):
\[
F^{(m)}((-1)^{m-1}t) = \prod_{n\geq 1}(1-t^n)^{-(-1)^{(m-1)n}n\dt_n^{(m)}}.
\]
\end{definition}

For the following theorem, note that the formula as stated in \cite{rein} has a typo.

\begin{thm}[Theorem 3.2 in \cite{rein}, see also Lemma 12 of \cite{ga number} and \cite{ga string math,ga ipmu}] \label{thm:Re}
We have
$\dt^{(m)}_n\in\NN$ and
\[
\dt^{(m)}_n = \frac{1}{n^2} \sum_{d|n} \mu\left(\frac{n}{d}\right) (-1)^{(m-1)(n-d)}\binom{dm-1}{d-1}.
\]
\end{thm}

Coming back to contributions of multiple covers to the $m^P_\beta$, we have the following unexpected result.

\begin{Proposition}[See Proposition \ref{prop:loopy}] \label{prop:loop}
Assume that the normalization map $[\pone\to C]\in\mse$ is infinitesimally rigid. Then
\[
\contr^{\bps}(l,C) = \dt^{(C.E-1)}_l.
\]

\end{Proposition}

\begin{proof}
The argument follows the same lines as the proof of the main result in \cite{ga number} and we therefore do not reproduce it here.
\end{proof}
Provided $C$ is nodal, using Proposition \ref{prop:loop} and Theorem \ref{thm:Re}, we find that $\contr^{\bps}(1,C) = \dt^{(C.E-1)}_1 = 1$ as expected. For the remainder of this section, $S$ will be a del Pezzo surface.

\begin{Proposition}\label{prop:lineconicmult}
Let $S$ be a del Pezzo surface and let $E$ be smooth anticanonical and general. Let $\beta$ be a line or conic class. Let $l\in\NN$ and let $P\in E(l\beta)$. If $P\in E(\beta)$, let $C$ be the unique curve of class $\beta$ maximally tangent to $E$ at $P$. Then the moduli space $\ol{\mmm}^P_{l\beta}(S,E)$ consists only of $l:1$ multiple covers of $C$ that are totally ramified at $P$, i.e.\ to the maps $[f : D\to C]$ with $f_*D=lC$ as cycles. Accordingly
\[
\mc{N}^P_{l\beta}(S,E)=\contr(l,C)=\begin{cases}
\frac{(-1)^{l-1}}{l^2} & \text{for line classes}, \\ 
\frac{1}{l^2} & \text{for conic classes}.
\end{cases}
\]
If $P\not\in E(\beta)$ however, then $\ol{\mmm}^P_{l\beta}(S,E)$ is empty and $\mc{N}^P_{l\beta}(S,E)=0$.
\end{Proposition}

\begin{proof}
Assume first that $P\in E(\beta)$. Let $[f : D\to S]\in\ol{\mmm}^P_{l\beta}(S,E)$. Assume that $\beta$ is a line class or a conic class of the form $h-e_i$. Then $\beta$ is extremal in the effective cone of $S$ and the only possible decomposition of $l\beta$ is as $\sum l_i \beta$, $\sum l_i =l$, hence $f_*D=lC$ as cycles. Thus $\ol{\mmm}^P_{l\beta}(S,E)$ consists only of $l:1$ multiple covers over $C$ and $\mc{N}^P_{l\beta}(S,E)=\contr(l,C)$.

If $\beta$ is a decomposable conic class, then as $\beta.K_S=-2$, $\beta=\eta_1+\eta_2$ for $\eta_1$ and $\eta_2$ line classes with necessarily $\eta_1.\eta_2=1$. For general $E$, that intersection point will be away from $E$, hence there is no maximal tangency log map with image curve reducible with one component in class $\eta_1$ and one component in class $\eta_2$ and the same argument as above applies.

If $P\not\in E(\beta)$ and $[f : D\to S]\in\ol{\mmm}^P_{l\beta}(S,E)$, then by the same argument as above, the image of $D$ has to be a curve of class $\beta$. But there are no such curves and therefore the moduli space is empty and $\mc{N}^P_{l\beta}(S,E)=0.$
\end{proof}

We prove the first part of Theorem \ref{thm1}:

\begin{Proposition} \label{prop:klineconic}
Let $S$ be a del Pezzo surface and let $E$ be smooth anticanonical and general.
Let $\beta$ be a line class or a conic class.
Let $l\in\NN$ and let $P\in E(l\beta)$.
Then
\[
m^P_\beta = \begin{cases}
1 & \text{if } l=1,\\
0 & \text{if } l\geq 2.
\end{cases}
\]
Consequently Conjectures \ref{conj_integrarity} and \ref{conj1} hold for multiples of line and conic classes.
\end{Proposition}

\begin{proof}
Assume first that $\beta$ is a line class. As $\beta.E=1$, $E(\beta)$ consists of only one point $P_0$, which is $\beta$-primitive. 
For that point, $m^{P_0}_\beta=1$ follows from Theorem \ref{thm:logcalc}(1). Denote by $C$ the (-1)-curve of class $\beta$ maximally tangent to $E$ at $P_0$. Let $l\geq 2$ and let $l'|l$. By Proposition \ref{prop:lineconicmult}, $\mc{N}^{P_0}_{l'\beta}(S,E)=\contr(l',C)$. In other words, the right hand side of \eqref{eq:recinv} (with $g=1$) is the same as the right hand side of Definition \ref{def:loopy}. Consequently $m^{P_0}_{l\beta}=\contr^{\bps}(l,C)=\dt_l^{(0)}$. We now apply Theorem \ref{thm:Re} and use a standard result about M\"obius functions:
\begin{align*}
\dt^{(0)}_l &= \frac{1}{l^2} \sum_{l'|l} \mu\left(\frac{l}{l'}\right) (-1)^{(-1)(l-l')}\binom{-1}{l'-1} \\
 &= \frac{1}{l^2} \sum_{l'|l} \mu\left(\frac{l}{l'}\right) (-1)^{(l-l')}(-1)^{l'-1} \\
 &= \frac{(-1)^{(l-1)}}{l^2} \sum_{l'|l} \mu\left(\frac{l}{l'}\right) = \frac{(-1)^{(l-1)}}{l^2} \sum_{l'|l} \mu\left(l'\right) = 0.
\end{align*}
Assume now that $P\in E(l\beta) \setminus E(\beta)$. Applying Proposition \ref{prop:lineconicmult} to  \eqref{eq:recinv}, we obtain that $m^P_{l\beta}=0$.

The proof for conic classes follows the same argument with the modification that, for $l\geq2$, $m^{P_0}_{l\beta}=\contr^{\bps}(l,C)=\dt_l^{(1)}$ and
\[
\dt_l^{(1)} = \frac{1}{l^2} \sum_{l'|l} \mu\left(\frac{l}{l'}\right) (-1)^{(1-1)(l-l')}\binom{l'-1}{l'-1} = \frac{1}{l^2} \sum_{l'|l} \mu\left(\frac{l}{l'}\right) = 0.
\]
\end{proof}

\subsection{Calculations for multiples of the hyperplane class}
\label{sec:evidencebig}

We summarize the results from this section, which proves the second part of Theorem \ref{thm1}:

\begin{Proposition}\label{prop:calchyperplane}
Let $S$ be a del Pezzo surface other than $\pone\times\pone$. Let $E$ be a general smooth anticanonical divisor and denote by $h$ the pullback of the hyperplane class. Then
\[
\begin{array}{llll}
m^P_h \, = 1 & \text{for all } P\in E(h), &
m^P_{2h} = 1 & \text{for all } P\in E(2h), \\
m^P_{3h} = 3 & \text{for all } P\in E(3h), &
m^P_{4h} = 16 & \text{for all } P\in E(4h).
\end{array}
\]
In particular Conjecture \ref{conj:big} holds for $S$ and curve classes $dh$, $1\leq d \leq 4$.
\end{Proposition}

We note that Proposition \ref{prop:calchyperplane} can also be deduced from the main result of \cite{Bou19b}. The approach we take here is by analysis of the moduli spaces, with a rather concrete description of the relevant curves. This illustrates the ``inner workings" of Definition \ref{def1} of log BPS numbers.

The rest of this section is devoted to proving Proposition \ref{prop:calchyperplane}. We proceed by analyzing the moduli spaces $\ol{\mmm}_{dh}^P(S,E)$ and calculating the contributions of each component to $\mc{N}_{dh}^P(S,E)$. First note that since $h.e_i=0$, $\ol{\mmm}_{dh}^P(S,E)$ remains unchanged when blowing down all exceptional divisors $e_i$. Accordingly, we may assume that $S=\ptwo$. Throughout this section, $d$ will be an integer $1\leq d \leq 4$.

Recall that $M_{dh}$ is the set of 
image cycles of genus $0$ stable log maps to $(\ptwo, E)$ of class $dh$ and
of maximal tangency. Recall that for $P\in E$, $M_{dh, P}=\{D\in M_{dh} \mid \mathrm{Supp}(D)\cap E=P\}$. 
By Proposition \ref{prop_description_log_maps_delpezzo}, $M_\beta = \coprod_{P\in E(dh)} M_{dh, P}$ and any element of $M_{dh,P}$ is a sum of rational curves maximally tangent to $E$, 
meeting $E$ at the same point.

We divide $E(dh)$ by arithmetic property of its points. 
Let 
\[
E(dh)_\prim:=\{P\in E(dh) \mid \hbox{$P$ is $dh$-primitive}\}. 
\]
In the current case of $S=\ptwo$, we have a decomposition 
\[
E(dh)=\coprod_{k | d} E(kh)_\prim. 
\]
Choosing as zero element $0\in E$ a flex point, 
we can identify the group $(E(dh),0)$ with $\ZZ/3d \times \ZZ/3d$ 
by Lemma \ref{lem:ebetagroup}. 
Under this identification, 
$E(kh)_\prim$ for a divisor $k$ of $d$ can be identified with 
\[
\left\{ P\in \ZZ/3d\times\ZZ/3d \; \Big{|} \; 
\hbox{$P$ is of order } \begin{array}{ll} k \text{ or } 3k & \text{if } 3\!\!\!\not| k \\ 
3k & \text{if } 3|k \end{array} \right\}.
\]
by Lemma \ref{lem:ptwo}. 

To give a rough description of $M_{dh,P}$, 
let 
\[
M_{dh, P}^\integ := \{C\in M_{dh, P} \mid \hbox{$C$ is integral}\}. 
\]
Let $P\in E(kh)_\prim$ and $d$ and $l$ multiples of $k$ with $l\leq d$. 
Then $P\in E(lh)$ and there may be a (necessarily finite) number of rational curves 
of degree $lh$ that are maximally tangent at $P$. 
Such curves, for different $l$, can be added together yielding elements of $M_{dh,P}$. 
For example, for $P\in E(h)=E(h)_\prim$, $M_{h,P}=M_{h,P}^\integ$ consists of 
the unique flex line $L_P$. 
Then $d \, L_P$ belongs to $M_{dh,P}$.

Recall that $j(E)$, the $j$-invariant of $E$, classifies smooth plane cubics defined over $\CC$ up to projective equivalence.
\begin{Proposition}\label{prop:lala}
The sets $M_{dh,P}$ are described as follows.

(1)
\[
\begin{array}{rll}
\text{If } P\in E(h), &\text{then } M_{h,P} &= \left\{ \, L_P \, \right\}, 
\text{where $L_P$ is the flex tangent.}
\end{array} 
\]

(2)
\[
\begin{array}{rll}
\text{If } P\in E(h), &\text{then } M_{2h,P} &=\left\{ \, 2 \, L_P \, \right\}. \\
\text{If } P\in E(2h)_\prim, &\text{then } M_{2h,P} & =M_{2h,P}^\integ = \left\{ \, C_P \, \right\} 
\text{ consists of a unique conic}. 
\end{array} 
\]

(3)
\[
\begin{array}{rll}
\text{If } P\in E(h), &\text{then } 
M_{3h,P} &=\left\{ \, 3 \, L_P\, \right\} \cup M_{3h,P}^\integ, 
\end{array} 
\]
where 
\[
\begin{array}{rll}
\text{if $j(E)\not=0$,} & \text{then }& 
\text{$M_{3h,P}^\integ$ consists of $2$ nodal cubics smooth at $P$, and} \\
\text{if $j(E)=0$,} & \text{then }& 
\text{$M_{3h,P}^\integ$ consists of $1$ cuspidal cubic smooth at $P$.}
\end{array} 
\]
\[
\begin{array}{rll}
\text{If } P\in E(3h)_\prim, &\text{then } 
M_{3h,P} &=M_{3h,P}^\integ  
\text{ consists of $3$ nodal cubics smooth at $P$.}
\end{array} 
\]

(4)
Let $E$ be general. 
\[
\begin{array}{rll}
\text{If } P\in E(h), &\text{then } 
M_{4h,P} &=
\left\{ 4 \, L_P \,\right\} \cup \left\{\, L_P + C \mid C\in M_{3h,P}^\integ \,\right\}
\cup M_{4h,P}^\integ \\
& & 
\text{ with $\big{|}M_{3h,P}^\integ\big{|}=2$ and $\big{|}M_{4h,P}^\integ\big{|}=8$}. \\
\text{If } P\in E(2h)_\prim, &\text{then } 
M_{4h,P} &=\left\{ \, 2 \, C_P \, \right\} \cup M_{4h,P}^\integ 
\text{ with $\big{|}M_{4h,P}^\integ\big{|}=14$}. \\
\text{If } P\in E(4h)_\prim, &\text{then } 
M_{4h,P} &=M_{4h,P}^\integ  
\text{ with $\big{|}M_{4h,P}^\integ\big{|}=16$}. 
\end{array}
\]
All curves in $M_{4h,P}^\integ$ are immersed and are smooth at $P$. 
\end{Proposition}

\begin{proof}
First note that $M_{dh, P}^\integ$ is exactly the set of $\AA^1$-curves, 
i.e. rational curves maximally tangent to $E$ at $P$. 

Noting that an element of $M_{dh, P}$ 
is a sum of elements of $M_{lh, P}$ for various $l$ by Proposition \ref{cor_description_log_maps}. 
it is easy to see that $M_{dh, P}$ is contained in the right hand side, 

For the other inclusion, the fact that a multiple of an $\AA^1$-curve 
is contained in $M_{dh, P}^\integ$ follows from \cite[Proposition 6.1]{GPS10}. 
For curves with $2$ components, we use Theorem \ref{thm:takmult}. 

Now we explain how to obtain the description of $M_{dh, P}^\integ$ given above. 

(1), (2), (3) 
For the $dh$-primitive points with $d=1, 2, 3$, the statements were proven in Section \ref{sec:calc}. 
The assertion on $M_{2h, P}$ for $P\in E(h)$ is easy. 

Assume that $P\in E(h)$, i.e.\ that $P$ is a flex point. 
Similarly to Section \ref{sec:logg1}, consider the pencil of degree 3 curves maximally tangent to $E$ at $P$. There is again a unique member $D$ that is singular at $P$. Necessarily, $D$ is the triple flex line. To obtain the universal family, we blow up $9$ times along the inverse images of $P$ in the strict transforms of $E$. This elliptic fibration (where the last exceptional divisor can be identified with the base) has Euler characteristic $12$. One fiber is a chain of 9 rational curves including the triple line. This fiber has Euler characteristic $10$. This means that in addition to the triple line, there are either two nodal cubics or one cuspidal cubic.

If there is a cuspidal cubic then in appropriate coordinates it is given by $Y^2Z-X^3$. We find that $E$ must then be given by $Y^2Z-X^3-aZ^3$, with $a\neq0$, which has $j(E)=0$. For $j(E)\neq0$, there therefore are two nodal cubics.

(4)
For $4h$-primitive points, the description of $M_{4h, P}^\integ$ is given 
in \cite[Theorem 2.1]{tak compl}. 
The general case was treated in \cite[Proposition 4.4]{tak mult} . 
Here we will just give an overview of the method. 

There is a triple cover $\pi: Y\to \ptwo$ totally ramified at $E$. 
It is easy to see that $Y$ is a del Pezzo surface of degree $6$, i.e. a cubic surface. 
For any $\AA^1$-curve $C$ of degree $d$ on $(\ptwo, E)$, 
$\pi^{-1}(C)$ splits into three $\AA^1$-curves of degree $d$ on $(Y, \pi^{-1}(E)_{\mathrm{red}})$. 
Let $C'$ be one of these. 
Luckily, for $d=4$, 
we have $p_a(C')\leq 1$ and 
we can find a blowdown $Y\to \ptwo$ 
which maps $C'$ to a conic or a cubic. 
The former case is easy, and the latter can be dealt with by the method of elliptic fibrations. 
After a careful analysis of the class of $C'$ and singular fibers of the elliptic fibrations 
in the case $p_a(C')=1$, 
which depends on the order of $P$, 
we obtain the result. 
\end{proof}

The multiplicities of reducible members are given by the following. 

\begin{Theorem}[Theorem 1.14 in \cite{CGKT3}, see also Theorem 1.4 in \cite{tak mult}]\label{thm:takmult}
Let $(X, D)$ be a pair consisting of 
a smooth surface and an effective divisor. 
Denote by $\ol{M}_\beta=\ol{M}_\beta(X, D)$ 
the moduli stack of 
maximally tangent genus $0$ basic stable log maps of class $\beta$ 
to the log scheme associated to $(X, D)$. 

Let $Z_1$ and $Z_2$ be proper integral curves on $X$ satisfying the following: 
\begin{enumerate}
\item 
$Z_i$ is a rational curve of class $\beta_i$ maximally tangent to $D$, 
\item
$(K_X+D).\beta_i=0$, 
\item
$Z_1\cap D$ and $Z_2\cap D$ consist of the same point $P\in D_\mathrm{sm}$, 
and 
\item
The normalization maps $f_i: \PP^1\to Z_i$ are immersive and 
$(Z_1.Z_2)_P=\min\{d_1, d_2\}$, 
where $d_i=D.Z_i$. 
\end{enumerate}

Write $d_1=de_1, d_2=de_2$ with $\gcd(e_1, e_2)=1$. 
Then there are $d$ stable log maps in $\ol{M}_{\beta_1+\beta_2}$ 
whose images are $Z_1\cup Z_2$, 
and they are isolated with multiplicity $\min\{e_1, e_2\}$.

When $X$ is projective and $(X,D)$ is log smooth, then these curves contribute $\min\{d_1, d_2\}$ to the log Gromov-Witten invariant $\mathcal{N}_{\beta_1+\beta_2}(X, D)$. 
\end{Theorem}

\begin{remark}
In Theorem \ref{thm:takmult}, the condition that $(Z_1.Z_2)_P=\min\{d_1, d_2\}$ means that $Z_1$ and $Z_2$ are assumed to intersect generically at $P$. 
We expect this condition is satisfied for general $E$ and $Z_1\not=Z_2$. 

If $d_1\not=d_2$, this follows from other assumptions 
since $Z_1$ and $Z_2$ are then smooth at $P$. 
In the case $d_1=d_2=d$, in an analytic coordinates $x, y$ near $P$ with $D=(y=0)$, 
we can write $Z_i=(y=a_ix^d+\dots)$. Then our assumption is that $a_1\not=a_2$. 

An example where this condition is obviously not satisfied is the case $Z_1=Z_2$. 
In this case, the space of log maps with image cycle $Z_1+Z_2$, 
as well as its contribution to the log Gromov-Witten invariant, 
is quite different (\cite[Proposition 6.1]{GPS10}).

\end{remark}

\begin{definition}

For $Z_1$ and $Z_2$ as in Theorem \ref{thm:takmult}, write
\[
\contr(Z_1,Z_2) := \min\{d_1, d_2\}.
\]

\end{definition}

\begin{Corollary} \label{cor:same}
Let $1\leq d \leq 4$ and $k|d$. Assume $E$ is general.
For $P,P' \in E(kh)_\prim$, $\mc{N}^P_{dh}(\ptwo, E) = \mc{N}^{P'}_{dh}(\ptwo, E)$.

\end{Corollary}

\begin{proof}

The description of the moduli spaces in Section \ref{sec:logGW} combined with Proposition \ref{prop:lala} and Theorem \ref{thm:takmult} can be used to prove this Corollary. However, a more conceptual proof for all degrees, based on the proof of Proposition \ref{prop:monodromic}, was given in \cite[Lemma 1.2.2]{Bou19b} and we refer to it.
\end{proof}

\begin{definition}
Assume $E$ is general and $k|d$. Set
\begin{align*}
\mc{N}^{k}_{dh}(\ptwo, E) &:= \mc{N}^P_{dh}(\ptwo, E), \\
m^{k}_{dh} &:= m^P_{dh}
\end{align*}
for $P\in E(kh)_\prim$. 
These numbers are well-defined by Proposition \ref{prop:lala} and Corollary \ref{cor:same}.
\end{definition}

\begin{Proposition}\label{prop:reallylala}
Assume $E$ is general. Let $1\leq d \leq 4$ and $k|d$, and take $P\in E(kh)_\prim$. Then
\begin{align*}
\mc{N}^{k}_{dh}(\ptwo, E) \, &=  \sum_{ l C \in M_{ dh,P} } \contr(l,C) + \sum_{ C_1 + C_2 \in M_{ dh,P} } \contr(C_1,C_2), \\ 
m^k_{dh} \, &=  \sum_{ l C \in M_{ dh,P} } \contr^{\bps}(l,C) + \sum_{ C_1 + C_2 \in M_{ dh,P} } \contr(C_1,C_2) \\
&=  \sum_{ l C \in M_{ dh,P} }  \dt^{\left(3d/l-1\right)}_l + \sum_{ C_1 + C_2 \in M_{ dh,P} } \contr(C_1,C_2),
\end{align*}
where the sums are over all elements of $M_{dh,P}$.
\end{Proposition}

\begin{proof}

The space $\ol{\mmm}^P_{dh}(\ptwo,E)$ has one isolated component of dimension $0$ for each integral cycle 
$C\in M_{dh,P}$. 
By assumption $C$ is immersed and hence by Proposition \ref{propocontro}(2), this component contributes $1=\contr(1,C)=\contr^{\bps}(1,C)$ to both $\mc{N}^{k}_{dh}(\ptwo, E)$ and $m^k_{dh}$.

Let $C_1+C_2\in M_{dh,P}$ 
with $C_1$ of degree $d_1=de_1$ and $C_2$ of degree $d_2=de_2$ for $\gcd(e_1,e_2)=1$ (and $d=d_1+d_2$). For each such cycle, by Theorem \ref{thm:takmult} $\ol{\mmm}^P_{dh}(\ptwo,E)$ has $3d$ isolated zero-dimensional components of length $\min\{e_1,e_2\}$ each. Together, they contribute $3d \min\{e_1,e_2\} = \contr(C_1,C_2)$  to both $\mc{N}^{k}_{dh}(\ptwo, E)$ and $m^k_{dh}$. 

The only components remaining are the components of $l:1$ covers over $C$ corresponding to image cycles 
$lC\in M_{dh,P}$. 
Such a component is of dimension $l-1$ and contribues $\contr(l,C)$ to $\mc{N}^{k}_{dh}(\ptwo, E)$ by \eqref{eq:BPS}, and $\contr^{\bps}(l,C)$ to $m^k_{dh}$ by \eqref{eq:recinv}.
\end{proof}

\begin{remark}
We note that to extend Proposition \ref{prop:reallylala} to $d>4$, one needs to compute the contributions to the log Gromov-Witten and log BPS invariants of components corresponding to more complicated image cycles.
\end{remark}

Using Propositions \ref{prop:lala} and \ref{prop:reallylala}, we can now prove Proposition \ref{prop:calchyperplane} by explicit calculation. We also include the calculation of log Gromov-Witten invariants for completeness. 
In the calculation of $m_{dh}^k$ etc., 
a point $P\in E(kh)_\prim$ is taken. 
We use notations $L_P$ for the flex tangent lines and $C_P$ for conics 
from Proposition \ref{prop:lala}. 
\begin{align*}
m^{1}_{h} &= \shN^{1}_{h}(\mathbb P^2,E) = \contr(1,L_P) = 1,\\
m^{1}_{2h} &= \contr^{\bps}(2,L_P) = \dt_2^{(2)} = 1, \\
\shN^{1}_{2h}(\mathbb P^2,E) &= \contr(2,L_P) = \frac{3}{4}, \\
m^{2}_{2h} &= \shN^{2}_{2h}(\mathbb P^2,E) =  \contr(1,C_P)  = 1, \\
m^{1}_{3h} &= \contr^{\bps}(3,L_P) + \sum_{C\in M_{3h, P}^\integ} \contr(1,C) \\
&= \dt^{(2)}_3 + \, 2 \cdot \dt^{(8)}_1 = 3, \\
\shN^{1}_{3h}(\mathbb P^2,E) &= \contr(3,L_P) + \sum_{C\in M_{3h, P}^\integ} \contr(1,C) =  3 + \frac{1}{9}, \\
m^{3}_{3h} &= \shN^{3}_{3h}(\mathbb P^2,E) =  \sum_{C\in M_{3h, P}} \contr(1,C) = 3, \\
m^{1}_{4h} &= \contr^{\bps}(4,L_P) + \sum_{C\in M_{3h, P}^\integ} \contr(L_P,C) + 
\sum_{C\in M_{4h, P}^\integ} \contr(1, C) \\
&= \dt^{(2)}_4 + \, 2 \cdot \min\{3,9\} + 8 = 16, \\
\shN^{1}_{4h}(\mathbb P^2,E) &= \contr(4,L_P) + \sum_{C\in M_{3h, P}^\integ} \contr(L_P,C) +  \sum_{C\in M_{4h, P}^\integ} \contr(1,C)  = 16+\frac{3}{16}, \\
m^{2}_{4h} &= \contr^{\bps}(2,C_P) + \sum_{C\in M_{4h, P}^\integ} \contr(1,C) \\
&= \dt^{(5)}_2 + 14 = 16, \\
\shN^{2}_{4h}(\mathbb P^2,E) &= \contr(2,C_P) + \sum_{C\in M_{4h, P}^\integ} \contr(1,C)
 = 16 + \frac{1}{4}, \\
m^{4}_{4h} &= \shN^{4}_{4h}(\mathbb P^2,E) = \sum_{C\in M_{4h, P}} \contr(1,C) = 16.
\end{align*}
This completes the proof of Proposition \ref{prop:calchyperplane}.

We finish by asking a question:

\begin{openque}
For each irreducible component $\ol{\mmm} \subset \msep$, is its contribution to $\nsep$ given by a DT invariant of some quiver? The next case to understand is the situation of Theorem \ref{thm:takmult}.

\end{openque}

\end{document}